\documentclass{amsart}
\usepackage[utf8]{inputenc}
\usepackage{format}

\title{Local character expansions and asymptotic cones over finite fields}
\begin{document}

\begin{abstract}
    We generalise Gelfand-Graev characters to $\mathbb R/\mathbb Z$-graded Lie algebras and lift them to produce new test functions to probe the local character expansion in positive depth.
    We show that these test functions are well adapted to compute the leading terms of the local character expansion and relate their determination to the asymptotic cone of elements in $\mathbb Z/n$-graded Lie algebras.
    As an illustration, we compute the geometric wave front set of certain toral supercuspidal representations in a straightforward manner.
\end{abstract}


\maketitle

\tableofcontents

%\footnotetext{Disclaimer: this is a preliminary draft meant to to accompany the second author's talk at the conference ``Anaparastaseis: Orbits, Hecke algebras, and Representations'', July 2023.}

\section{Introduction}
Let $\mathbb{G}$ be a connected reductive group defined over a $p$-adic field $F$ with Lie algebra $\textfrak g$. 
We will require the residue characteristic $F$ to be larger than some constant depending on the absolute root datum of $\mathbb G$ and the ramification index of $F/\mathbb Q_p$ (c.f. hypotheses \ref{hyp:exp}, \ref{hyp:p}, and \ref{hyp:bilinear-form} for details).
Fix an algebraic closure $\bar F$ of $F$, let $G = \mathbb G(F)$, $\mathfrak g = \textfrak g(F)$ and let $\mathcal N(\mathfrak g^*)$ denote the nilpotent cone of the linear dual of $\mathfrak g$.
Let $C_c^\infty(\mathfrak g)$ denote compactly supported smooth functions on $\mathfrak g$ and let $\FT:C_c^\infty(\mathfrak g)\to C_c^\infty(\mathfrak g^*)$ denote the Fourier transform.
For a co-adjoint orbit $\mathcal O^*$ let $\mu_{\mathcal O^*}$ denote the associated $G$-invariant measure on $\mathfrak g^*$.

Let $(\pi,V)$ be an irreducible smooth representation of $G$ with distribution character $\Theta_\pi$. 
To each co-adjoint orbit $\mathcal O^*\in \mathcal N(\mathfrak g^*)/G$ we may attach a complex number $c_{\mathcal O^*}(\pi) \in \CC$ such that
\begin{equation}\label{eq:localcharacter}\Theta_{\pi}(f) = \sum_{\mathcal O^*\in \mathcal N(\mathfrak g^*)/G} c_{\mathcal O^*}(\pi) \FT(\mu_{\mathcal O^*})(f\circ\exp)\end{equation}
for any $f\in C_c^\infty(G)$ supported on a sufficiently small neighborhood of $1$. 
Here $\exp$ is the exponential map or a suitable variant. 
The formula (\ref{eq:localcharacter}) is called the \emph{local character expansion} of $\pi$ and is due to Howe \cite{Howe1974} for $GL_n$ and Harish-Chandra \cite{HarishChandra1999} in general. 
The maximal open set $U_\pi\ni 1$ on which Equation \ref{eq:localcharacter} holds is called the domain of validity for the local character expansion.

The coefficients $c_{\mathcal{O}^*}(\pi)$ are expected to encode significant information about the representation $\pi$. For instance, the Hiraga--Ichino--Ikeda conjecture \cite{HII} links $c_{0}(\pi)$ to the adjoint gamma factor of the $L$-parameter of $\pi$, when $\pi$ is square integrable. For unipotent representations with real infinitesimal character, the monodromy of the $L$-parameter of the Aubert-Zelevinsky dual of $\pi$ is determined by the leading terms of the local character expansion as demonstrated by \cite{ciubotaru2023wavefront}.
The leading terms of the local character expansion, with respect to the topological closure order,
$$\WF(\pi):=\max\{\mathcal O^*:c_{\mathcal O^*}(\pi)\ne0\},$$
are called the {\it wave front set}. 
%However, for orbits between these extremes, little is known, largely due to the complexity of computing the local character expansion.

To compute the coefficients of the local character expansion, one evaluates both sides of equation \ref{eq:localcharacter} with appropriately chosen test functions and solves the resulting system of equations. Although evaluating the precise values of orbital integrals on test functions may be challenging, if one has for each co-adjoint nilpotent orbit $\mathcal O^*$ a test function $f_{\mathcal O^*}$ satisfying
\begin{enumerate}
    \item the support of $f_{\mathcal O^*}$ is contained in $U_\pi$,
    \item $\FT(\mu_{\mathcal O^*})(f_{\mathcal O'^*})\ne0$ only if $\mathcal O'^*\le \mathcal O^*$ and,
    \item $\FT(\mu_{\mathcal O^*})(f_{\mathcal O^*})\ne0$,
\end{enumerate}
then the resulting system of equations is upper triangular. 
This simplifies the system of equations and gives rise to a particularly simple way to compute the wave front as
$$\WF(\pi) = \max\{\mathcal O^*: \mathcal O^*\in \mathcal N(\mathfrak g^*)/G, \Theta_\pi(f_{\mathcal O^*})\ne 0\}.$$

We refer to \cite{MW87,barmoy,debacker,murnaghan} for important examples of such test functions.
%Notably \cite{tsai2023geometric} used the test functions in \cite{debacker2} in order to compute the wave front set of his counter example to the Moeglin-Waldspurger conjecture, and \cite{okada2022wave front,,ciubotaru2023wave front,ciubotaru2024wave front} used the test functions in \cite{barmoy} to compute wave front sets of certain depth-0 representations.

%Noting that $U_\pi$ depends only on the depth of $\pi$, 
The aim of this paper is to generalise the test functions defined in \cite{barmoy} so that they can be applied to representations with positive depth. We pay special attention to the implications for the wave front set.
We are motivated by the results in \cite{ciubotaru2022wavefront, ciubotaru2024wavefront}, which demonstrate the effectiveness of Barbasch and Moy's test functions in computing the wave front set for supercuspidal representations of depth 0.

Let us now state the main results of this paper in more detail.
%%We are in particular motivated by the results in \cite{ciubotaru2022wavefront,ciubotaru2024wavefront} which demonstrate how effective Barbasch and Moy's test functions are in computing the wavefront set for supercupsidal representations of depth 0.

%As observed in \cite{debacker}, the domain of validity $U_\pi$ depends on the depth of $\pi$. Consequently, we aim to develop a set of test functions suitable for all representations of a fixed depth. 

%In this paper we will generalise the test function from \cite{barmoy} to the positive depth setting. 

\subsection{Statement of main results}
Let $\mathscr B(\mathbb G,F)$ denote the (enlarged) Bruhat-Tits building for $G$.
For $x\in \mathscr B(\mathbb G,F)$ and $r\ge 0$ (resp. $r\in \mathbb R$) let $G_{x,r}$ (resp. $\mathfrak g_{x,r}$, $\mathfrak g^*_{x,r}$) denote the Moy-Prasad filtration subgroups for $G$ (resp. $\mathfrak g$, $\mathfrak g^*$) associated to the point $x$.
Let
$$\mathfrak g_{r^+} = \bigcup_{x\in \mathscr B(\mathbb G,F)}\mathfrak g_{x,r^+}, \quad G_{r^+} = \bigcup_{x\in \mathscr B(\mathbb G,F)}G_{x,r^+}$$
and define $\underline{\mathfrak g}_{x,r} := \mathfrak g_{x,r}/\mathfrak g_{x,r^+}, \underline{\mathfrak g}^*_{x,r} := \mathfrak g^*_{x,r}/\mathfrak g^*_{x,r^+}, \underline G_{x,r} := G_{x,r}/G_{x,r^+}, \underline G_x := \underline G_{x,0}$.
Let 
$$\mathscr L_{x,r}:\mathcal N(\underline {\mathfrak g}_{x,r}^*)/\underline G_x\to \mathcal N(\mathfrak g^*)/G$$
denote the lifting map of co-adjoint orbits defined by DeBacker in \cite{debacker}.
For $\underline{\mathcal O}^*\in \mathcal N(\underline{\mathfrak g}_{x,r}^*)/\underline G_x$ the orbit $\mathscr L_{x,r}(\underline{\mathcal O}^*)$ is characterised by the property that it is the minimal-dimensional co-adjoint nilpotent orbit with a representative for $\underline{\mathcal O}^*$ in $\mathfrak g^*_{x,r}$.

Our first result is the existence of test functions satisfying conditions (1)-(3) from the introduction.
By \cite{waldspurger95,debacker2}, if an irreducible smooth representation $(\pi,G)$ has depth $r$ then $U_\pi \supseteq G_{r^+}$ and so it suffices to produce test functions supported on $G_{r^+}$.

\begin{theorem}[Lemma \ref{lem:eval}, Proposition \ref{prop:orbitalint}]
    Let $r\ge 0$.
    For every $x\in \mathscr B(\mathbb G,F)$ and $\underline{\mathcal O}^*\in \mathcal N(\underline{\mathfrak g}^*_{x,-r})/\underline G_x$ the test functions $f_{x,r,\underline{\mathcal O}^*}$ defined in Definition \ref{d:test-function} satisfy
    \begin{enumerate}
        \item $\supp(f_{x,r,\underline{\mathcal O}^*}) \subset G_{r^+}$,
        \item $\FT(\mu_{\mathcal O^*})(f_{x,r,\underline{\mathcal O}^*}\circ\exp)\ne0$ only if $\mathscr L_{x,-r}(\underline{\mathcal O}^*)\le \mathcal O^*$ and,
        \item $\FT(\mu_{\mathscr L_{x,-r}(\underline {\mathcal O}^*)})(f_{x,r,\underline{\mathcal O}^*}\circ\exp)\ne0.$
    \end{enumerate}
    
    Note that for any $\mathcal O^*\in \mathcal N(\mathfrak g^*)/G$ there exists $x\in \mathscr B(\mathbb G,F)$ and $\underline{\mathcal O}^*\in \mathcal N(\underline{\mathfrak g}_{x,-r})/\underline G_x$ such that $\mathscr L_{x,-r}(\underline {\mathcal O}^*) = \mathcal O^*$ so we can define a test function for any nilpotent co-adjoint orbit.
\end{theorem}

Our next main theorem relates the determination of
$$\max\{\mathscr L_{x,-r}(\underline{\mathcal O}^*): \underline{\mathcal O}^*\in \mathcal N(\underline{\mathfrak g}^*_{x,-r})/\underline G_x, \Theta_\pi(f_{x,r,\underline{\mathcal O}^*})\ne 0\}$$
for a fixed $x\in \mathscr B(\mathbb G,F)$ and $r>0$ (not necessarily equal to the depth of $\pi$), to a problem in finite field geometry.

To simplify the exposition of this introduction, we will present our main results under the assumption that $r$ is rational (in which case we can replace $\mathbb R/\mathbb Z$-gradings with $\mathbb Z/n$-gradings).
%Since all representations have rational depth, this problem is most interesting when $r$ is rational.
%Furthermore, the geometric problem can be stated in terms of $\mathbb Z/n$-graded Lie algebras in this case (as opposed to $\mathbb R/\mathbb Z$-graded Lie algebras which are less familiar). 
%For these reasons, we will assume that $r$ is rational when presenting our main results in this introduction. 
%The general case, where $r>0$ is arbitrary, is treated in the main body of the paper.
%Moreover, the statements of our main results are slightly simpler in the rational case, and for that reason we will make this assumption when stating our main results in this introduction.
%We refer to the main body of the paper for the case where $r>0$ is arbitrary.

\begin{definition}
    Let 
    $$\mathfrak g = \bigoplus_{i\in \mathbb Z/n}\mathfrak g_i$$
    be a $\mathbb Z/n$-graded Lie algebra over $\mathbb F_q$.
    \begin{enumerate}
        \item We say an element $x\in \mathfrak g_i$ \emph{degenerates} to a nilpotent element $e\in\mathfrak g_i$ if $e$ can be completed to an $\mathfrak {sl}_2$-triple $e\in \mathfrak g_i, h\in \mathfrak g_0, f\in \mathfrak g_{-i}$ such that
        $$x\in e+\bigoplus_{j\le0}\mathfrak g_i(j)$$
        where $\mathfrak g_i(j)$ denotes the grading induced by the triple (c.f definition \ref{hyp:graded}.(4)).
        \item If $\mathfrak g$ admits a non-degenerate $\ad$-invariant bilinear form $B$ on $\mathfrak g$ such that
        $$B(\mathfrak g_i,\mathfrak g_j) = 0, \quad \text{ if } i+j\ne 0$$
        then $B$ induces an isomorphism 
        $$\eta_B:\mathfrak g_i^*\xrightarrow{\sim}\mathfrak g_{-i}$$
        and we say $\alpha\in \mathfrak g_i^*$ degenerates to a nilpotent $\beta\in \mathfrak g_i^*$ if $\eta_B(\alpha)$ degenerates to $\eta_B(\beta)$.
        \item For a subset $\Psi\subset \mathfrak g_i^*$ we define \emph{the rational asymptotic cone} of $\Psi$, denoted $\mathrm{cone}(\Psi)$, to be the set of nilpotent elements in $\mathfrak g_i^*$ that are degenerated from an element in $\Psi$.
    \end{enumerate}
\end{definition}

Fix a point $x\in \mathscr B(\mathbb G,F)$ and a rational number $r=m/n>0$.
If we normalise the valuation $\nu$ on $F$ so that $\nu(F^\times) = \mathbb Z$ then the quotients $\underline{\mathfrak g}_{x,s}$ of the Moy-Prasad filtration are $\mathbb Z$-periodic.
Thus
$$\underline {\mathfrak g}^r_x:=\bigoplus_{i\in \mathbb Z/n}\underline{\mathfrak g}_{x,ri}$$
is a $\mathbb Z/n$-graded Lie algebra defined over the residue field of $F$, and we can talk about degenerations and asymptotic cones of elements in $\underline{\mathfrak g}_{x,r}^*$. Denote by $C(\underline{\mathfrak g}_{x,r}^*)$ the space of complex-valued functions on $\underline{\mathfrak g}_{x,r}^*$ and let $f^*$ denote the complex conjugate of a function $f\in C(\underline{\mathfrak g}_{x,r}^*)$.

\begin{theorem}\label{thm:A}[Theorem \ref{thm:shallow}]
    Let $(\pi,V)$ denote a smooth irreducible representation of $G$ of depth $d(\pi)$ and fix a point $x\in \mathscr B(\mathbb G,F)$ in the building.
    Let $r\in \mathbb Q_{>0}$ and $r\ge d(\pi)$.
    Define $\underline \pi_{x,r}:=\pi^{G_{x,r^+}}$ and let ${\theta}_{\pi,x,r}$ denote the character of $\underline\pi_{x,r}$ as a representation of $\underline G_{x,r}\cong \underline {\mathfrak g}_{x,r}$.
    
    Then
    $$\Theta_\pi(f_{x,r,\underline{\mathcal O}^*}) \ne 0$$
    if and only if $\underline{\mathcal O}^*$ lies in the rational asymptotic cone of the support of $\FT(\theta^*_{\pi,x,r})\in C(\underline{\mathfrak g}_{x,r}^*)$.

    In particular
    $$\WF(\pi) = \max\bigcup_{x\in \mathscr C}-\mathscr L_{x,-r}(\mathrm{cone}(\supp(\FT({\theta}_{\pi,x,r})))).$$
\end{theorem}

Our next result gives a bound on the wave front set.
\begin{definition}\label{def:geo-cone}
    Let 
    $$\mathfrak g = \bigoplus_{i\in \mathbb Z/n}\mathfrak g_i$$
    be a $\mathbb Z/n$-graded Lie algebra over $\overline{\mathbb F}_q$ and $G$ be a reductive group defined over $\overline{\mathbb F}_q$ with Lie algebra $\mathfrak g_0$ and a compatible linear action on $\mathfrak g$ that preserves the grading.

    For an element $\alpha\in \mathfrak g^*_i$ we define the \emph{geometric asymptotic cone of $\alpha$} to be 
    $$\overline\cone(\alpha) := \mathcal N(\mathfrak g^*_i) \cap \overline{\{\lambda^2g.\alpha:g\in G,\lambda\in \overline{\mathbb F}_q^\times\}} \quad \subset \mathcal N(\mathfrak g^*_i)$$
    where the closure is taken with respect to the Zariski topology.
\end{definition}

Let $\Fun$ denote the maximal unramified extension of $F$ and $G^{un}=\mathbb G(\Fun),\mathfrak g^{un}=\textfrak g(\Fun)$. 
Note that the composition factors $\underline{\mathfrak g}_{x,r}$ of the Moy-Prasad filtration are the $\mathbb F_q$-rational points of the vector space $\underline{\mathfrak g}^{un}_{x,r}$ over $\overline{\mathbb F}_q$.
Similarly $\underline G_x$ is the $\mathbb F_q$-rational points of the reductive group $\underline G_{x}^{un}$ defined over $\mathbb F_q$ and so for a rational number $r=m/n$, the pair 
$$\left(\underline G_x^{un},\bigoplus_{i\in \mathbb Z/n}\underline{\mathfrak g}^{un}_{x,ir}\right)$$
satisfy the hypotheses of definition \ref{def:geo-cone}.

\begin{theorem}[Corollary \ref{cor:bound}]
    The rational asymptotic cone is a subset of the geometric asymptotic cone.
    
    %Keep the notation of theorem \ref{thm:A}.
    In particular, in the notation of theorem \ref{thm:A}, the wave front set of $\pi$ is contained in the closure of
    $$\bigcup_{x\in \mathscr C}\mathscr L^{un}_{x,-r}(\overline\cone(\supp(\FT({\theta}_{\pi,x,r})))) \quad \subset \mathfrak g^{un,*}.$$
\end{theorem}

Our last theorem gives a closed form for the rational and geometric asymptotic cone of an element in the $0$-graded piece of a $\mathbb Z/n$-graded Lie algebra.
In this setting the grading is superfluous and so we drop it from the statement of the theorem. 
\begin{theorem}[Theorem \ref{thm:ratwfungraded}]
    Let $G$ be a reductive group defined over $\overline{\mathbb F}_q$ with Lie algebra $\mathfrak g$ and a Frobenius endomorphism $\sigma$.
    Write $?^\Fr$ for the $\Fr$-fixed points of $?$.
    %We identify $\mathfrak g$ with the $\overline{\mathbb F}_q$-points of $\mathfrak g$ and we write $F:\mathfrak g\to \mathfrak g$ for the Frobenius map induced by the $\mathbb F_q$-structure.
    %We write $\mathfrak g^F$ for the Frobenius fixed points of $\mathfrak g$ (which are the $\mathbb F_q$ points of $\mathfrak g$).
    
    Let $x\in \mathfrak g$ and let $x = x_s+x_n$ denote its Jordan decomposition.
    Then 
    $$\overline\cone(x) = \overline{\Ind_{L}^G(L.x_n)}$$
    where $L=Z_{G}(x_s)$, $\overline{(-)}$ denotes the Zariski closure, and $\Ind$ refers to Lusztig-Spaltenstein induction.

    Moreover, if $x\in \mathfrak g^\Fr$ then $(\Ind_L^G(L.x_n))^\Fr$ is non-empty and contains degenerations of $x$. 
    In particular
    $$\cone(x) \cap \Ind_{L}^G(L.x_n)\ne\emptyset.$$
\end{theorem}

Finally, to demonstrate the usage of these results we show in section \ref{sec:adler} that the geometric wave front set of certain toral supercuspidal representations consists of a single nilpotent orbit and we give an explicit formula for it. 

\medskip

\noindent{\bf Acknowledgements.}
%The authors would like to thank Maxim Gurevich, Uri Onn, and Lucas Mason-Brown.
We would like to thank Maxim Gurevich, Uri Onn and Lucas Mason-Brown for stimulating discussions about the orbit method. We also thank the referee for the careful reading of the paper and for the helpful comments. This research was partially supported by the EPSRC grant EP/V046713/1 (2021).

\

\noindent {\bf Notation.}
Let $F$ be a finite extension of $\mathbb Q_p$ with valuation $\nu$ normalised so that $\nu(F^\times) = \mathbb Z$, $\bar F$ an algebraic closure for $F$, and $F^{un}$ the maximal unramified extension of $F$ in $\bar F$.
Write $\nu$ for the unique extension of the valuation of $F$ to any algebraic extension of $F$.
Let $\mathbb F_q$ be the residue field for $F$.
The residue field of $F^{un}$ is an algebraic closure for $\mathbb F_q$ and we denote it by $\overline{\mathbb F}_q$.
Write $\mathfrak O_F$ for the ring of integers for $F$, and $\mathfrak p_F$ for the maximal ideal of $\mathfrak O_F$.
Fix a non-trivial character $\chi_0:\mathbb F_p\to \mathbb C^\times$ and let $\chi = \chi_0\circ \mathrm{Tr}_{\mathbb F_q/\mathbb F_p}:\mathbb F_q\to \mathbb C^\times$.
By lifting along the quotient map $\mathfrak O_F\to \mathbb F_q$ we can extend to obtain a character $F\to \mathbb C^\times$ which we will also denote $\chi$, which is trivial on $\mathfrak p_F$.
Denote $\mathbb N_0=\mathbb N\cup \{0\}.$

\section{Preliminaries}
\label{sec:git}
Let $G$ be a connected reductive group defined over $\mathbb F_q$ and let $V = \mathbb A_{\mathbb F_q}^N$ be a $\mathbb F_q$-rational representation of $G$.
Let $\check{V}$ denote the contragredient representation of $V$.
We will identify $G$, $V$ and $\check V$ with their $\overline{\mathbb F}_q$-points and we will denote by $\Fr$ the Frobenius induced by the $\mathbb F_q$-structure.
We will write $?^\Fr$ for the $\Fr$-fixed points of $?$.

\begin{definition}
    $\hphantom{ }$
    \begin{enumerate}
        \item Let $\mathbb G_m$ act on $V$ via $a.v = a^2v$.
        Since $V$ is a linear representation of $G$, the $G$ action commutes with the $\mathbb G_m$ action and so $V$ is naturally a $G \times \mathbb G_m$ representation.
        Write $\tilde {G}$ for $G\times \mathbb G_m$.
        \item An element $v\in V$ is called \emph{nilpotent} if the Zariski closure of $Gv$ contains $0$.
        Write $\mathcal N(V)$ for the set of nilpotent elements of $V$ and call this the nilpotent cone of $V$.
        \item Write $C(V^\Fr)$ for the space of complex-valued functions on $V^\Fr$, and $C(V^\Fr/G^\Fr)$ for the subspace of $C(V^\Fr)$ consisting of $G^\Fr$-invariant functions.
        \item Let $C(V^\Fr,+)$ denote the set of functions $f\in C(V^\Fr)$ that arise as characters of representations of $V^\Fr$ (viewed as an additive group), and let $C(V^\Fr/G^\Fr,+)$ denote the subset of $C(V^\Fr,+)$ consisting of $G^\Fr$-invariant functions.
        \item For a function $f\in C(V^\Fr)$ define
        $$\mathrm{FT}(f):\check V^\Fr \to \mathbb C, \quad \alpha \mapsto \frac{1}{q^{N/2}}\sum_{v\in V^\Fr}\chi(\alpha(v))f(v).$$
        Similarly, define the Fourier transform for functions on $\check V^\Fr$.
        \item For functions $f,g\in C(V)$ define
        $$\langle f,g\rangle_{V} := \frac{1}{q^{N}}\sum_{v\in V}\overline{f(v)}g(v)$$
        where $\overline{z}$ denotes the complex conjugate of $z\in \mathbb C$. Similarly, define an inner product on complex-valued functions on $\check V$.
        \item Identify the Pontryagin dual of $V^\Fr$ with $\check V^\Fr$ by identifying $\alpha \in \check V^\Fr$ with the character $\chi \circ \alpha:V^\Fr\to \mathbb C^\times$.
        For $\psi:V^\Fr\to \mathbb C^\times$ an additive character, let $\alpha_\psi$ denote the element of $\check V^\Fr$ corresponding to $\psi$ so $\psi = \chi \circ \alpha_\psi$.
    \end{enumerate}
\end{definition}
%Let $G = \mathbb G(\mathfrak f), \tilde G= \widetilde {\mathbb G}(\mathfrak f), V = \mathbb V(\mathfrak f), \check V = \check{\mathbb V}(\mathfrak f)$.
%We also write $\mathbb G, \widetilde{\mathbb G}, \mathbb V,\check{\mathbb V}$ for their $\overline{\mathfrak f}$-points.
The Fourier transform satisfies the following basic identities:
$$\langle \FT(f),\FT(g)\rangle_{\check V} = \langle f, g\rangle_{V}, \quad \FT(\FT(f))(v) = f(-v).$$
%

%We also have that 
%%
%$$\psi = q^{N/2}\cdot \FT(1_{\alpha_\psi}),\quad \FT(\psi) = q^{N/2}1_{-\alpha_\psi}.$$
%%

%The following lemma follows immediately from the fact that $V$ is abelian and so all irreducible characters of $V$ are linear.
The following lemma gives a characterisation of additive characters in terms of the Fourier transform.
The proof is straightforward and omitted.
\begin{lemma}
    \label{lem:coadjcharexp}
    Let $f\in C(V^\Fr)$. 
    \begin{enumerate}
        \item $f\in C(V^\Fr,+)$ if and only if $\FT(f)$ takes values in $q^{N/2}\mathbb N_{0}$.
        \item $f\in C(V^\Fr/G^\Fr,+)$ if and only if $f = \FT(h)$ where $h = \sum_{\mathcal O^*\subseteq \check V^\Fr/G^\Fr}c_{\mathcal O^*}(f)1_{\mathcal O^*}$ for some $c_{\mathcal O^*}(f)\in q^{N/2}\mathbb N_0$.
    \end{enumerate}
\end{lemma}

We call a $G^\Fr$-invariant character reducible if it can be written as the sum of two non-zero $G^\Fr$-invariant characters and irreducible otherwise.
Let $\Pi(V^\Fr/G^\Fr,+)$ denote the set of irreducible $G^\Fr$-invariant characters of $V^\Fr$.
%It follows from Lemma \ref{lem:coadjcharexp} that $\Pi(V^F/G^F,)$ is precisely the set functions of the form $\FT(q^{N/2}1_{\mathcal O^*})$ for some $G^F$-orbit $\mathcal O^*$ of $\check V^F$.
\begin{lemma}
    \label{lem:irrchar}
    For a orbit $\mathcal O^*\in \check V^\Fr/G^\Fr$ let $\chi_{\mathcal O^*}$ denote the $G$-invariant representation $\FT(q^{N/2}1_{-\mathcal O^*})$.
    %Note that $-\mathcal O^*$ is still a $G$-orbit. We simply introduce the minus sign so that $\chi_{\mathcal O^*}$ has the property $\FT(\chi_{\mathcal O^*}) = q^{N/2}1_{\mathcal O^*}$.
    The map 
    $$\check V/G \to \Pi(V/G,+), \quad \mathcal O^*\mapsto \chi_{\mathcal O^*}$$
    is a bijection and satisfies
    \begin{enumerate}
        \item $\FT(\chi_{\mathcal O^*}) = q^{N/2}1_{\mathcal O^*}$,
        \item $\overline{\chi_{\mathcal O^*}} = \chi_{-\mathcal O^*}$.
    \end{enumerate}
\end{lemma}
\begin{proof}
    The map is a bijection by lemma \ref{lem:coadjcharexp}.(2).
    The identities $(1),(2)$ are straightforward computations.
\end{proof}

\begin{definition}
    Let $S\subset V$.
    Define the \textit{geometric asymptotic cone} of $S$ to be
    $$\overline{\cone}(S):=\mathcal N(V) \cap \overline{\tilde G.S} \quad \subset \mathcal N(V)$$
    where the closure is taken with respect to the Zariski topology.
    %\Fror $f\in C(V)$ define the asymptotic cone of $f$ to be
    %%
    %$$\cone(f):=\mathrm{cone}(\supp(\FT(f))) \quad \subset \mathcal N(\mathbb V^*).$$
    %%
    %%Dual to this notion we define the \textit{wave front set} of $f$ to be
    %%
    %$$\WF(f) := \cone(\FT(f)) \quad \subseteq \mathcal N(\check {\mathbb V}).$$
    %%
    %\Fror a discussion of the role of the asymptotic cones in the local character expansions for reductive groups over local fields and for their relevance to the global coherence of wave front sets, see \cite[\S3]{adamsvogan2021}.
\end{definition}
%We have the following trivial lemma.
%\begin{lemma}
    %\label{lem:asywfadditive}
    %Let $f = \sum_{\mathcal O^*\in \mathcal N(V)/G}c_{\mathcal O^*}(f)\chi_{\mathcal O^*}$ be a $G$-equivariant character of $V$ and $\bullet\in \{tr,asy\}$.
    %Then
    %%
    %$$\WF(f) = \bigcup_{\mathcal O^*:c_{\mathcal O^*}(f)\ne 0}\WF(\chi_{\mathcal O^*}).$$
    %%
%\end{lemma}
%This reduces our study of wave front sets of $G$-invariant characters to the irreducible case.
%So a basic question to ask is do the maps $\WF^{tr},\WF:\Pi(V/G,+) \cong \check {V}/G\to \mathcal N(\check{\mathbb V})$ have nice descriptions?
\section{The wave front set of characters in $\mathbb R/\mathbb Z$-graded Lie algebras}
\label{sec:vinburglevy}
Let $\mathfrak g$ be a Lie algebra defined over $\overline{\mathbb F}_q$ with a compatible $\mathbb R/\mathbb Z$-grading 
$$\mathfrak g = \bigoplus_{r\in \mathbb R/\mathbb Z}\mathfrak g_r$$
i.e. $[\mathfrak g_r,\mathfrak g_s]\subset \mathfrak g_{r+s}$ where addition is performed in $\mathbb R/\mathbb Z$.
Such gradings include $\mathbb Z/n$-gradings via the embedding $\mathbb Z/n\to \mathbb R/\mathbb Z, 1\mapsto 1/n$, and $\mathbb Z$-gradings via $\mathbb Z\to \mathbb R/\mathbb Z, n\mapsto \alpha n$ where $\alpha$ is irrational.

%Let $n\in \mathbb Z$ and let $\mathfrak g$ be a Lie algbra defined over $\overline{\mathbb F}_q$ with a compatible $\mathbb Z/n$-grading 
%
%$$\mathfrak g = \bigoplus_{i\in \mathbb Z/n}\mathfrak g_i$$
%
%i.e. $[\mathfrak g_i,\mathfrak g_j]\subset \mathfrak g_{i+j}$.
Let $p$ be the characteristic of $\mathbb F_q$.

In this section we will assume that $\mathfrak g$ satisfies the following hypotheses.
\begin{hypothesis}\label{hyp:graded}
    $\hphantom{ }$
    \begin{enumerate}
        \item $\mathfrak g_0$ is the Lie algebra of a reductive group $G$ and each $\mathfrak g_r$ carries an action of $G$ compatible with the adjoint action of $\mathfrak g_0$;
        \item $G$ and $\mathfrak g$ are equipped with compatible (and grading preserving) Frobenius endomorphisms $\Fr$;
        \item every nilpotent element $e\in \mathfrak g_r$ can be completed to an $\mathfrak {sl}_2$-triple 
        $$e\in \mathfrak g_r, \quad h\in \mathfrak g_0,\quad f\in \mathfrak g_{-r}.$$
        Moreover there exists a cocharacter $\lambda:\mathbb G_m\to G$, unique up to multiplication by a cocharacter of the center of $G$ with $0$ differential, such that the image of $d\lambda$ is equal to the space spanned by $h$, and
        $$\text{ if } \Ad(\lambda(t))v = t^iv, \text{ then } |i|\le p-3 \text{ and } \ad(h) v = iv.$$
        If $e,h,f\in \mathfrak g^\Fr$ then $\lambda$ can be chosen (and so we will henceforth always assume) to commute with $\Fr$.
        \item $\mathfrak g$ is equipped with a $G$-invariant, non-degenerate, symmetric bilinear form 
        $$B:\mathfrak g\times\mathfrak g \to \overline{\mathbb F}_q$$
        compatible with $\Fr$, and satisfying 
        $$B(\mathfrak g_r,\mathfrak g_{s})\ne 0 \text{ only if } r+s = 0.$$
    \end{enumerate}
\end{hypothesis}
\begin{definition}\label{def:graded}
    $\hphantom{ }$
    \begin{enumerate}
        \item We call a $\lambda$ satisfying hypothesis \ref{hyp:graded}.(3) an adapted cocharacter for the $\mathfrak {sl}_2$-triple $(e,h,f)$.
        \item The bilinear form $B$ restricts to a non-degenerate bilinear form $B:\mathfrak g_{-r}\times \mathfrak g_r\to \overline{\mathbb F}_q$ and hence yields an isomorphism $\eta_B:\mathfrak g_r^*\to \mathfrak g_{-r}$.
        Since $B$ is compatible with $\Fr$, $\eta_B$ restricts to an isomorphism $(\mathfrak g_r^*)^\Fr\to \mathfrak g_{-r}^\Fr$.
        \item Since $\ad(e)^p=0$ for all $e\in \mathcal N(\mathfrak g_0)$, the exponential map $\exp:\mathcal N(\mathfrak g_0)\to G$ is defined on the nilpotent elements of $\mathfrak g_0$.
        \item Given an $\mathfrak {sl}_2$-triple $e\in \mathfrak g_r,h\in \mathfrak g_0,f\in \mathfrak g_{-r}$ and an adapted cocharacter $\lambda$, define
        $$\mathfrak g^\lambda_{r}(j) := \{x\in\mathfrak g_r:\Ad(\lambda(t)).x = t^jx, t\in \mathbb G_m\}.$$
        This grading only depends on the $\mathfrak {sl}_2$-triple since conjugating by a central cocharacter does nothing and so when there is no cause for confusion we will drop the $\lambda$ superscript.
    \end{enumerate}
\end{definition}

%We remark that this setting is more general than that arising in Vinberg-Levy theory.
%In that setting the grading comes from an automorphism of order $n$, but it requires $\mathfrak g$ to be reductive and $n$ to be co-prime to $q$. 
%These are both assumptions we do not need to make.
%Let $N_i:=\dim_{\mathbb F_q}(\mathfrak g_i^F)$.

%We will need the following hypothesis about nilpotent elements.

%\begin{hypothesis}\label{lem:graded-sl2}
    %Every nilpotent element $e\in \mathcal N(\mathfrak g_i)$ can be completed to an $\mathfrak {sl}_2$ triple 
    %%
    %$$e\in \mathfrak g_i, h\in \mathfrak g_0, f\in \mathfrak g_{-r}.$$
    %%
%
    %Moreover, if $F(e) = e$ then $h,f$ can also be chosen so that $F(h)=h, F(f) = f$.
%\end{hypothesis}
%
%\begin{lemma}
    %...All such triples are conjugte by an element of $G^F$...
%\end{lemma}

\begin{lemma} \label{lem:graded-sl2}
    Let $e\in \mathcal N(\mathfrak g_r)$ and let $\Theta_r(e)$ denote the set of all $\mathfrak sl_2$-triples $(e,h,f)$ such that 
    $$e\in \mathfrak g_r, h\in \mathfrak g_0,f\in \mathfrak g_{-r}.$$
    %
    %Let $U$ denote the unipotent radical of $G_e$.
    
    Then 
    \begin{enumerate}
        \item there exists a connected unipotent subgroup $U$ of $G_e$ that acts simply transitively on $\Theta_i(e)$.\label{en:graded-sl2-1}
    \end{enumerate}
    If additionally $e\in \mathfrak g_i^\Fr$ then
    \begin{enumerate}
        \item[(2)] $\Theta_i(e)^\Fr$ is non-empty \label{en:graded-sl2-2}
        %and the adapted cocharacter $\lambda$ can be chosen so that it commutes with $\Fr$,
        \item[(3)] $U^\Fr$ acts simply transitively on $\Theta_i(e)^\Fr$.\label{en:graded-sl2-3}
    \end{enumerate}
\end{lemma}
\begin{proof}
    $(1)$ By hypothesis \ref{hyp:graded}, $\Theta_i(e)$ non-empty. 
    Let $(e,h,f)\in \Theta_i(e)$ and $\lambda$ be an adapted cocharacter for this $\mathfrak {sl}_2$-triple.
    
    The cocharacter $\lambda$ induces a bi-grading 
    $$\mathfrak g = \bigoplus_{r\in \mathbb R/\mathbb Z,j\in \mathbb Z}\mathfrak g_r(j).$$
    Let 
    $$\mathfrak u_0^e = \mathfrak g_0^e \cap [\mathfrak g_{-r},e].$$
    It is contained in $\oplus_{j>0}\mathfrak g_0(j)$ and so is a nilpotent subalgebra of $\mathfrak g_0$.
    In particular the exponential map is defined on $\mathfrak u^e_0$ and so the proof of \cite[Lemma 3.4.7]{CM} yields that 
    $$U:=\exp(\mathfrak u^e_0) \quad \subset Z_G(e)$$
    acts simply transitively on $H+\mathfrak u_0^e$.

    The proof of \cite[Theorem 3.4.10]{CM} then gives that $U$ acts simply transitively on $\Theta_i(e)$.
    The group $U$ is connected because it is the image of the exponential map of a Lie subalgebra.
    
    $(2)$ Since $e\in \mathfrak g_i^\Fr$, it follows that $\mathfrak u_0^e$ and hence $U$ is $\Fr$-stable.
    The existence of a $\Fr$-fixed point then follows from Lang's theorem.
    %The fact that $\lambda$ can be chosen to be compatible with $\Fr$ follows from \cite[Appendix A.3]{debacker}.

    $(3)$ Suppose $(e,h,f),(e,h',f')\in \Theta_i(e)^\Fr$.
    Then there exists a $u\in U$ such that 
    $$\Ad(u).(e,h,f) = (e,h',f').$$
    Applying $\Fr$ to both sides we get that
    $$\Ad(\Fr(u)).(e,h,f) = (e,h',f').$$
    By simple transitivity it follows that $\Fr(u) = u$ as required.
\end{proof}
%%Need to show \Fr-stable sl2 triples exist for any \Fr stable nilpotent orbits
\begin{definition}
    We say an element $x\in \mathfrak g_r^\Fr$ \emph{degenerates} to a nilpotent element $e\in\mathfrak g_r^\Fr$ if $e$ can be completed to an $\mathfrak {sl}_2$-triple $e\in \mathfrak g_r^\Fr, h\in \mathfrak g_0^\Fr, f\in \mathfrak g_{-r}^\Fr$ with adapted cocharacter $\lambda$ such that 
    $$x\in e+\bigoplus_{j\le0}\mathfrak g_r(j).$$

    For a subset $S\subset \mathfrak g_r^\Fr$ we define the \emph{rational asymptotic cone} of $S$ to be the set of nilpotent elements in $\mathfrak g_r$ that are degenerated from an element in $S$.
    We denote the rational asymptotic cone of $S$ by $\cone(S)$.
\end{definition}

\subsection{Graded generalised Gelfand-Graev representations}
\label{sec:gggg}
In this section we define graded generalised Gelfand-Graev representations.
% (or 4G representations for short).
These are the $\mathbb R/\mathbb Z$-graded analogues of the generalised Gelfand-Graev characters studied by Kawanaka.
%We assume that the characteristic is larger than $3(h-1)$ where $h$ is the Coxeter number for $\mathbb G$ in order to use the Jacobson-Morosov theorem in positive characteristic.

\begin{definition}
    \label{def:gggg}
    Let $r\in \mathbb R/\mathbb Z$ and $\alpha \in \mathcal N((\mathfrak g_r^{*})^\Fr)$.
    %Let $x_\alpha$ denote the corresponding element of $\mathfrak g_{-r}$.
    By lemma \ref{lem:graded-sl2}.(2) we can complete $\eta_B(\alpha)\in \mathfrak g_{-r}^\Fr$ to an $\mathfrak {sl}_2$-triple, 
    $$\eta_B(\alpha)\in \mathfrak g_{-r}^\Fr, \quad h\in \mathfrak g_0^\Fr, \quad f \in \mathfrak g_r^\Fr$$
    and associate to this triple an adapted cocharacter $\lambda$ that commutes with $\Fr$.
    The action of $\lambda$ induces a bi-grading 
    $$\mathfrak g = \bigoplus_{r\in \mathbb R/\mathbb Z,j\in \mathbb Z}\mathfrak g_{r}(j).$$
    Write $\mathfrak g_r(\ge j)$ for $\bigoplus_{k\ge j}\mathfrak g_r(k)$.
    Define $\Gamma_{\alpha}$ to be the function on $\mathfrak g_r^\Fr$ given by
    \begin{equation}\label{e:Gamma-alpha}
    \Gamma_{\alpha}(x) := q^{N_r}\sum_{g\in G^\Fr,\Ad(g)(x)\in \mathfrak g_r(\le -1)}\chi(-\alpha(\Ad(g)(x)))
    \end{equation}
    where $N_r = \dim_{\mathbb F_q}\mathfrak g_r^\Fr$.
    We also introduce the notation 
    $$\Sigma_\alpha = \alpha + \eta_B^{-1}(\mathfrak g_{-r}(\le 0)).$$
    It depends on the choice of $\mathbb F_q$-rational $\mathfrak {sl}_2$, but since all such triples are conjugate by an element in $G^\Fr$ we supress this from the notation without harm.
    Similarly we write $\Sigma_{G^\Fr.\alpha}$ for $\Sigma_{\alpha}$ without harm.
\end{definition}
\begin{proposition}\label{prop:gggg}
    Let $\alpha\in \mathcal N((\mathfrak g_r^*)^\Fr)$.
    Then
    \begin{enumerate}
        \item $\Gamma_\alpha$ only depends on the $G^\Fr$-orbit of $\alpha$;
        \item $\supp(\Gamma_\alpha)\subset \mathcal N(\mathfrak g_r)$;
        \item $\Gamma_\alpha\in C(\mathfrak g_r^\Fr/G^\Fr,+)$;
        \item $\langle\chi_{\mathcal O'^*},\Gamma_{\alpha}\rangle \ne 0$ iff $\mathcal O'^*\cap \Sigma_{\alpha}\ne \emptyset$.
    \end{enumerate}
\end{proposition}
\begin{proof}
    $(1)$ First note from lemma \ref{lem:graded-sl2} that the definition of $\Gamma_{\alpha}$ does not depend on the choice of graded $\mathfrak {sl}_2$-triple.
    It is then clear that the definition then only depends on the $G^\Fr$-orbit of $\alpha$.

    $(2)$ This follows from the fact that $\mathfrak g(\le -1)$ is a nilpotent Lie subalgebra and hence consists of nilpotent elements.
    %It is also clear that $\Gamma_\alpha$ only depends on the coadjoint orbit $\mathcal O^* = G.\alpha$ so we write $\Gamma_{\mathcal O^*}$ instead of $\Gamma_\alpha$.

    $(3)$ It is clearly $G^\Fr$-equivariant.
    Moreover we have that
    \begin{align*}
        q^{-N_i/2}\cdot\FT(\Gamma_{\alpha})(\beta) &= \sum_{x\in \mathfrak g_r^\Fr, g\in G^\Fr,\Ad(g)(x)\in \mathfrak g_r(\le -1)}\chi(\beta(x))\chi(-\alpha(\Ad(g)(x))) \\
        &= \sum_{x\in \mathfrak g_r^\Fr, g\in G^\Fr,\Ad(g)(x)\in \mathfrak g_r(\le -1)}\chi((\Ad^*(g)\beta-\alpha)(\Ad(g)x)) \\
        &= \sum_{x\in \mathfrak g_r(\le -1)^\Fr, g\in G^\Fr}\chi((\Ad^*(g)\beta-\alpha)(x)) \\
        &= \sum_{x\in \mathfrak g_r(\le -1)^\Fr, g\in G^\Fr}\psi_{g,\alpha,\beta}(x)
    \end{align*}
    where $\psi_{g,\alpha,\beta} = \chi\circ (\Ad^*(g)\beta-\alpha)$. 
    It is clear this is a character of $\mathfrak g_r(\le -1)^\Fr$ and so the sum $\sum_{x\in \mathfrak g_r(\le -1)^\Fr}\psi_{g,\alpha,\beta}(x)$ is $0$ unless $\Ad^*(g)(\beta)-\alpha$ is $0$ on $\mathfrak g_r(\le -1)^\Fr$.
    The bilinear form $B$ restricts to a perfect pairing on $\mathfrak g_{-r}(\le -1)^\Fr\times \mathfrak g_{r}(\ge 1)^\Fr$ and so $\Ad^*(g)(\beta)-\alpha$ is $0$ on $\mathfrak g_r(\le -1)^\Fr$ if and only if $\Ad(g)(\eta_B(\beta))-\eta_B(\alpha)\in \mathfrak g_{-r}(\le 0)^\Fr$.
    Therefore
    \begin{align}
        \label{eq:ftgggg}
        \FT(\Gamma_{\alpha})(\beta) &= q^{N_i/2}\cdot \#\mathfrak g_r(\le -1)^\Fr \cdot \#\{g\in G^\Fr : \Ad(g)(\eta_B(\beta)) \in \eta_B(\alpha) +  \mathfrak g_{-r}(\le 0)\}.
    \end{align}
    Since $\FT(\Gamma_{\alpha})$ takes values in $q^{N_i/2}\mathbb N_{0}$, $\Gamma_{\alpha}$ is a $G$-equivariant character of $\mathfrak g_r^\Fr$ by lemma \ref{lem:coadjcharexp}.

    $(4)$ This follows immediately from equation \ref{eq:ftgggg}.
\end{proof}
\begin{definition}\label{d:Gamma-O} In light of part (1) of this proposition we write $\Gamma_{\mathcal O^*}$ for $\Gamma_{\alpha}$ in (\ref{e:Gamma-alpha}), where $\alpha$ is any element of $\mathcal O^*$ and $\mathcal O^*\in (\mathfrak g_r^*)^\Fr/G^\Fr$.
\end{definition}

\subsection{The wave front set}

%When the choice of base point does not matter we will write $\Sigma_{\mathcal O^*}$ for $\Sigma_{\alpha}$ for some $\alpha \in \mathcal O^*$.
%This is justified by the fact that all $\mathfrak {sl}_2$-triples for $\mathcal O^*$ are conjugate by an element of $G_0$ and $\Sigma_{\mathcal O^*}$ is well defined up to conjugation by $G_0$.
%\begin{definition}
    %\label{def:order}
    %Let $\mathcal O_1^*,\mathcal O_2^*\subseteq N(\mathfrak g_r^*)/G_0$.
    %Define the rational order on coadjoint orbits of $\mathfrak g_r$ to be
    %%
    %$$\mathcal O_1^*\le \mathcal O_2^* \text{ if } \Sigma_{\mathcal O_1^*}\cap \mathcal O_2^* \ne \emptyset$$
    %%
    %for some $\alpha_1\in \mathcal O_1^*$.
%\end{definition}
%TODO: prove this defines a partial order
\begin{definition}
    %\item For $\mathcal O^*\in \mathcal N(g_i^*)$ define
    %%
    %$$\Sigma_{\mathcal O^*} := e_\alpha + \mathfrak g_{-r}(\le 0)$$
    %%
    %where $\alpha\in \mathcal O^*$ and $(e_\alpha,h_\alpha,f_\alpha)\in \Theta_i^F$.
    %This depends on the choice of $\alpha$ and $\mathfrak {sl}_2$-triple for $e_\alpha$, but by lemma \ref{lem:graded-sl2} it is well defiend up to conjugation by $G_0^F$ and so we suppress this from the notation.
    For a function $f$ on $\mathfrak g_r^\Fr$, define the \emph{wave front set} of $f$ to be
    $$\overline{\WF}(f) = \{\mathcal O^*\in (\mathfrak g_r^*)^\Fr/G^\Fr: \langle f,\Gamma_{\mathcal O^*}\rangle_{\mathfrak g_r} \ne 0\}.$$
\end{definition}
%For $\chi_{\mathcal O'^*}:\mathfrak g_r\to \mathbb C$, $\mathcal O'^*\in \mathfrak g_r^*$ an irreducible $G_0$-invariant character of $\mathfrak g_r$ we have 
%
%$$\langle \chi_{\mathcal O'^*},\Gamma_{\mathcal O^*}\rangle = q^{N_i/2}\langle 1_{\mathcal O'^*},\FT(\Gamma_{\mathcal O^*})\rangle$$
%
%which is non-zero if and only if $\mathcal O'^*$ intersects $\Sigma_{\mathcal O^*}$.

%\begin{lemma}
    %\label{lem:ratwfadditive}
    %Let $f = \sum_{\mathcal O^*\in \mathcal N(V)/G}c_{\mathcal O^*}(f)\chi_{\mathcal O^*}$ be a $G$-equivariant character of $V$.
    %Then
    %%
    %$$\WF^{rat}(f) = \max\{\mathcal O^*:c_{\mathcal O^*}(f)\ne 0, \mathcal O^*\in \WF^{rat}(\chi_{\mathcal O^*})\}.$$
    %%
%\end{lemma}
%\begin{proof}
    %Since $c_{\mathcal O'^*}(f)\ge 0$ for all $\mathcal O'^*$, $\langle f,\Gamma_{\mathcal O^*}\rangle_{\mathfrak g_r}\ne 0$ if and only if $\langle \chi_{\mathcal O'^*},\Gamma_{\mathcal O^*}\rangle_{\mathfrak g_r}\ne 0$ for some $\mathcal O'^*$ with $c_{\mathcal O'^*}(f)\ne 0$.
    %The result follows.
%\end{proof}
%
\begin{lemma}\label{lem:wf-cone-relation}
    If $f\in C(\mathfrak g_r^\Fr/G^\Fr,+)$ then 
    $$\overline\WF(f) = \cone(\supp(\FT(f))).$$
\end{lemma}
\begin{proof}
    By lemma \ref{lem:coadjcharexp}, 
    $$f = \sum_{\mathcal O*\in (\mathfrak g_r^*)^\Fr/G^\Fr}c_{\mathcal O^*}\chi_{\mathcal O^*}, \quad c_{\mathcal O^*} \in \mathbb N_0.$$
    Thus $\langle f,\Gamma_{\mathcal O^*}\rangle\ne 0$ if and only if $\langle \chi_{\mathcal O'^*},\Gamma_{\mathcal O^*}\rangle\ne0$ for some $\mathcal O'\in \supp(\FT(f))$.

    By proposition \ref{prop:gggg}.(4), $\langle \chi_{\mathcal O'^*},\Gamma_{\mathcal O^*}\rangle\ne0$ if and only if $\mathcal O'^*$ intersects $\Sigma_{\mathcal O^*}$.
    But by definition this is the case if and only if $\mathcal O\subset \cone(\mathcal O'^*)$.
    The result follows.
\end{proof}

\begin{prop}
    \label{prop:slice}
    $\hphantom{ }$
    \begin{enumerate}
        \item Let $\alpha\in \mathfrak g_r^*$ and $\beta\in \mathcal N(\mathfrak g_r^*)$.
        Then $\beta \in\overline{\cone}(\alpha)$ if and only if $G.\alpha \cap \Sigma_\beta\ne 0$.
        \item Let $\mathcal O^*$ be a $G^\Fr$-orbit of $(\mathfrak g_r^*)^\Fr$.
        Then 
        $$\overline{\WF}(\chi_{\mathcal O^*}) \subset \overline{\cone}(\mathcal O^*).$$
    \end{enumerate}
\end{prop}
\begin{proof}
    $(1)$ $(\Leftarrow)$ %Identify $\mathfrak g_r^*$ with $\mathfrak g_{-r}$.
    %Let $e_\alpha,e_\beta\in \mathfrak g_{-r}$ denote the elements corresponding to $\alpha,\beta$ respectively under the identification.
    Let $\eta_B(\beta)\in \mathfrak g_{-r},h_\beta \in \mathfrak g_0, f_\beta \in \mathfrak g_r$ be an $\mathfrak {sl}_2$-triple for $\eta_B(\beta)$ and $\lambda$ be an adapted cocharacter.
    %Let $\lambda_{\beta}:\mathbb G_m\to G$ denote the 1-parameter subgroup determined by $h_\beta$ so that $\Ad(\lambda_\beta(t)).x = t^{-j}x$ for $x\in \mathfrak g_{-r}(j)$.
    By assumption $G.\alpha \cap \Sigma_\beta\ne \emptyset$.
    Let $\eta_B(\beta)+z$ with $z = \sum_{i\le0}z_i\in \mathfrak g_{-r}(\le0)$ be in $\eta_B(G.\alpha\cap \Sigma_\beta)$.
    Let $\widetilde{\lambda_{\beta}}:\mathbb G_m\to \widetilde{G}$ denote the 1-parameter subgroup given by $t\mapsto (\lambda_\beta(t),t)$.
    Then 
    $$\widetilde{\lambda_\beta}(t).(\eta_B(\beta)+\sum_{i\le0}z_i) = \eta_B(\beta) + \sum_{i\le 0}t^{2-i}z_i \to \eta_B(\beta), \text{ as }t\to 0.$$
    %
    %TODO: make sure this makes sense over $\overline{\mathbb F}_q$.
    Since $\eta_B(\beta)+z$ also lies in the $G$ orbit of $\eta_B(\alpha)$ we get that $\eta_B(\beta)$ lies in $\overline{\widetilde{G}.\eta_B(\alpha)}$.
    $(\Rightarrow)$ Continue the notation from the previous part of the proof. Consider the graded Slowdowy slice at $\eta_B(\beta)$ 
    $$\mathfrak s_{\beta} := \eta_B(\beta)+ c_{\mathfrak g_{-r}}(f_\beta).$$
    Since $\mathfrak g = [\mathfrak g, \eta_B(\beta)] \oplus c_{\mathfrak g}(f_\beta)$ we get by taking graded parts that 
    $$\mathfrak g_{-r} = [\mathfrak g_0, \eta_B(\beta)] \oplus c_{\mathfrak g_{-r}}(f_\beta).$$
    Thus the graded Slodowy slice intersects the orbit $G.\eta_B(\beta)$ at $\eta_B(\beta)$ transversally.
    %TODO: algebraic geometry justificaiton for next sentence (see Barbasch Moy).
    Since the map $\mu:G\times \mathfrak s_\beta\to \mathfrak g_{-r}$ is smooth and $G,\mathfrak s_\beta$ are both non-singular, there exists an open set $\mathcal U$ containing $\eta_B(\beta)$ in the image of $\mu$.
    Thus for any $x\in \mathfrak g_{-r}$, if $G.x \cap \mathcal U \ne \emptyset$ then $G.x \cap \mathfrak s_{\beta}\ne \emptyset$.
    Since $\eta_B(\beta) \in\overline{\widetilde{G}.\eta_B(\alpha)}$, we have $\mathcal U\cap \widetilde{G}.\eta_B(\alpha)\ne \emptyset$.
    Let $x$ lie in the intersection.
    Since $x\in \mathcal U$, $G.x \cap \eta_B(\Sigma_\beta)\ne\emptyset$ so after conjugating by $G$ appropriately, we obtain an element $y\in \widetilde{G}.\eta_B(\alpha) \cap \eta_B(\Sigma_\beta)$.
    Write $y = a^2 \Ad(g)\eta_B(\alpha)$ for some $a\in \overline{\mathbb F}_q^\times$ and $g\in G$.
    Thus $\Ad(g)\eta_B(\alpha) \in a^{-2}\eta_B(\beta) + c_{\mathfrak g_{-r}}(f_\beta)$. But then $\Ad(\lambda(a^{-1})).y \in \eta_B(\beta) + c_{\mathfrak g_{-r}}(f_\beta)$ and so $G.\eta_B(\alpha) \cap \eta_B(\Sigma_\beta) \ne \emptyset$ as required.

    $(2)$ By lemma \ref{lem:wf-cone-relation},
    $$\overline{\WF}(\chi_{\mathcal O^*}) = \cone(\mathcal O^*).$$
    By part (1), we have that 
    $$\cone(\mathcal O^*) \subset \overline{\cone}(\mathcal O^*).$$
\end{proof}

\section{The wave front set of $0$-graded characters}
\label{sec:ungraded}
In this section we provide a closed form for the wave front set of $\chi_{\mathcal O^*}$ for a co-adjoint orbit $\mathcal O^*\subset \mathfrak g_0^*$ in the $0$-graded piece. 

Under the assumptions of hypothesis \ref{hyp:graded}, restricting to the $0$-graded piece of a $\mathbb R/\mathbb Z$-graded Lie algebra is equivalent to considering an ungraded reductive Lie algebra.
%and relate it to the rational wave front set.
Thus, for this section, we will keep the notation of the previous section but we will omit the $\mathbb R/\mathbb Z$ grading.
%We start with a short detour on the theory of sheets.
%The results are due to Borho \cite{borho} in characteristic 0, and \cite{premet_stewart_2018} in good characteristic.

\subsection{Borho's theory of $G$-Sheets}

Let \[\mathfrak g^{(m)} = \{x\in \mathfrak g: \dim G.x = m\}.\] This is a $G$-stable locally closed subvariety of $\mathfrak g$ and the irreducible components are called {\it sheets}.
%Note that sheets need not be disjoint.

\begin{theorem}
    \label{thm:sheets}
    Suppose the hypotheses in \ref{hyp:graded} are in effect (viewing the ungraded case as a $\mathbb R/\mathbb Z$-graded Lie algebra concentrated in degree $0$).
    \begin{enumerate}
        \item \cite[\S2.3,\S2.5]{premet_stewart_2018} For an element $x\in \mathfrak g$ let $x = x_s + x_n$ be the Jordan decomposition of $x$.
        Let $L_x = C_{G}(x_s)$ and $\mathfrak l_x = c_{\mathfrak g}(x_s)$. The group $L_x$ is a Levi subgroup of $G$.
        %Call the pair $(\mathbb L_x,\mathbb L_x.x_n)$ the decomposition pair of $x$.
        %We say two elements $x,x'\in \mathfrak g$ are in the same decomposition class if they have conjugate decomposition pairs.
        %Every sheet contains a unique dense decomposition class.
        Define the map 
        $$\mathscr N:\mathfrak g/G \to \mathcal N(\mathfrak g)/G, \quad \mathscr N(x):=\Ind_{\mathfrak l_x}^{\mathfrak g}L_x.x_n$$
        where $\Ind$ denotes Lusztig-Spaltenstein induction.
        The map $\mathscr N$ is constant on any sheet and every sheet contains a unique nilpotent orbit given exactly by the value of $\mathscr N$ on the sheet.
        \item \cite[Appendix B]{tsai} Fix a sheet $\mathcal S$ and let $\phi=(e,h,f)$ be an $\mathfrak{sl}_2$-triple for $\mathscr N(\mathcal S)$.
        Let 
        $$\mathfrak s_\phi = e+c_{\mathfrak g}(f)$$
        be the associated Slodowy slice and let $G_{\phi}$ be the joint centraliser of $e,h,f$ in $G$.
        Every $G$-orbit in $\mathcal S$ intersects $\mathfrak s_\phi$ in a non-empty discrete set and two points $x,x'\in\mathfrak s_\phi$ are $G$-conjugate if and only if they are $G_\phi$-conjugate.
    \end{enumerate}
\end{theorem}

\subsection{The asymptotic cone and rational points}
%\begin{lemma}
    %\label{lem:slices}
    %Let $\phi= (e,h,f)$ be an $\mathfrak{sl}_2$-triple in $\mathfrak g$ and let $\mathfrak g(i)$ denote the $i$-eigenspace of $\ad(h)$.
    %Let 
    %%
    %$$U^\phi = \exp\left(\bigoplus_{i\le -2}\mathfrak g(i)\right).$$
    %%
    %Then the map
    %%
    %$$e+\mathfrak g^f \times U^\phi\to e+\bigoplus_{i\le 0}\mathfrak g(i), \quad (z,u)\mapsto \Ad(u)z$$
    %%
    %is an isomorphism and induces a bijeciton on $F$-fixed points.
%\end{lemma}

\begin{prop}
    \label{prop:dilatedsuppungraded}
    Let $x\in \mathfrak g$. Then
    $$\overline{\cone}(x) = \overline{\mathscr N(x)}.$$
\end{prop}
\begin{proof}
    $(\supseteq)$ Let $\mathcal S$ be a sheet containing $x$ and $e,h,f$ be an $\mathfrak {sl}_2$-triple for $\mathscr N(x)$.
    By theorem \ref{thm:sheets}.(2) $G.x \cap \mathfrak s \ne \emptyset$.
    %By lemma \ref{lem:slices}, this means $G.x$ intersects $\Sigma_{G.e}$.
    Since 
    $$c_{\mathfrak g}(f)\subseteq \bigoplus_{i\le 0}\mathfrak g(i)$$
    we have $\mathfrak s \subset \Sigma_{G.e}$ and so $G.x$ intersects $\Sigma_{G.e}$.
    By Proposition \ref{prop:slice}.(1), $\mathscr N(x)$ lies in $\overline{\cone}(x)$.
    
    $(\subseteq)$ Let $\mathfrak z_x$ denote the center of $\mathfrak l_x$. 
    Note that $x_s\in \mathfrak z_x$ so $x\in \mathfrak z_x + \mathbb L_x. x_n$ and that $\mathfrak z_x + \mathbb L_x. x_n$ is a $\mathbb G_m$-invariant set.
    Thus $\widetilde{G}.x \subseteq G.(\mathfrak z_x + \mathbb L_x. x_n)$ and so $\overline{\widetilde{G}.x} \subseteq \overline{G.(\mathfrak z_x + \mathbb L_x. x_n)}$.
    By \cite[Lemma 2.5 (b)]{borho} we have $\mathcal N \cap \overline{G(\mathfrak z_x + \mathbb L_x.x_n)} \subseteq \overline{\mathscr N(x)}$. The desired inclusion follows.

\end{proof}

\begin{definition}
    \begin{enumerate}
        \item Let $\Theta\subseteq \mathfrak g^{\oplus 3}$ denote the set of all $\mathfrak {sl}_2$-triples $(e,h,f)\in \mathfrak g^{\oplus 3}$.
        This set is determined by the equations
        $$[h,e] = 2e, \quad [h,f] = -2f, \quad [e,f] = h$$
        and so $\Theta$ is a closed $G$-invariant subvariety of $\mathfrak g^{\oplus 3}$.
        \item For a nilpotent orbit $\mathcal O$ let $\Theta(\mathcal O)$ denote the subset of $\Theta$ consisting of elements $(e,h,f)$ with $e\in \mathcal O$.
        \item For an $\mathfrak {sl}_2$-triple $\phi = (e,h,f)\in \Theta$ let 
        $$\mathfrak s_\phi:= e + c_{\mathfrak g}(f)$$
        denote the corresponding Slodowy slice.
        \item For an element $x\in \mathfrak g$ let $\Theta(x)$ denote the set of all $\mathfrak {sl}_2$-triples $\phi$ such that $\phi\in \Theta(\mathscr N(x))$ and $x\in \mathfrak s_\phi$.
        This set is always non-empty by theorem \ref{thm:sheets}.(2).
        %Since $g.x\in \mathfrak s_\phi$ if and only if $x\in s_{g^{-1}\phi}$, this set is non-empty by Proposition \ref{prop:slice} and Lemma \ref{lem:dilatedsuppungraded}.
        \item Call an element $x\in \mathfrak g$ good if $G_x^\circ G_\phi = G_x G_\phi$ for some $\phi\in \Theta(x)$.
    \end{enumerate}
    %The Jacobson-Morosov theorem states that the $G$-orbits on $\Theta$ are in one-to-one correspondence with $\mathcal N(\mathfrak g)/G$ via projection onto the first factor.
\end{definition}

%Clearly $g\mathfrak s_\phi = \mathfrak s_{g\phi}$ where $g\phi = (ge,gh,gf)$.
%When $\phi\in \Theta^F$ we have $\mathfrak s^F_\phi= e + c_{\mathfrak g}(f)$.

\begin{lemma}
    Let $x$ be semisimple or nilpotent. Then $x$ is good.
\end{lemma}
\begin{proof}
    Let $\phi\in \Theta(x)$.
    If $x$ is semisimple then $G_x = G_x^\circ$ and so it is trivially good.
    If $x$ is nilpotent, then $\mathscr N(x) = G.x$ and $G.x \cap \mathfrak s_\phi = \{x\}$ so $\phi(e) = x$.
    By \cite[Remark 3.7.5 (ii)]{CM}, the inclusion map $G_\phi \to G_x$ induces an isormorphism of component groups.
    In particular $G_\phi\to G_x/G_x^\circ$ is surjective and so $G_x^\circ G_\phi = G_x G_\phi$.
\end{proof}

%\begin{lemma}
    %\label{lem:stableind}
    %If $x\in \mathfrak g$ then $\mathscr N(x)$ is $F$-stable
%\end{lemma}
%\begin{proof}
%\end{proof}

\begin{lemma}\label{lem:rational-points}
    Let $x\in \mathfrak g^\Fr$ be good.
    Then $\Theta(x)^\Fr$ is non-empty.
\end{lemma}
\begin{proof}
    By proposition \ref{prop:dilatedsuppungraded}, $\overline{\mathscr N(x)} = \mathcal N(\mathfrak g) \cap \overline{\widetilde{G}.x}$. 
    But both $\mathcal N(\mathfrak g)$ and $\overline{\widetilde {G}.x}$ are $\Fr$-stable so the intersection must be too.
    It follows that $\mathscr N(x)$ is $\Fr$-stable and hence so is $\Theta(x)$.
    We will show that $\Theta(x)$ is a single $G_x^\circ$-orbit and hence by Lang's theorem, contains an $\Fr$-fixed point.
    Let $\phi\in \Theta(x)$.
    Any other element $\phi'\in \Theta(x)$ is of the form $g\phi$ since $\Theta(\mathscr N(x))$ is a single $G$-orbit.
    But $x\in \mathfrak s_{g\phi}$ if and only if $g^{-1}x\in \mathfrak s_\phi$.
    By theorem \ref{thm:sheets}.(3) there exists an element $h\in G_\phi$ such that $g^{-1}x = hx$.
    Thus $gh\in G_x$ and so $g\in G_x G_\phi$.
    But $G_x^\circ G_\phi = G_x G_\phi$ and so $\phi'\in G_x^\circ \phi$ as required.
\end{proof}

\begin{remark}
    After an earlier version of this preprint was posted on the arxiv, Tsai in \cite[Theorem B.4]{tsai} proved that, in fact, all elements are good.
    Thus we can drop the assumption that $x$ is good in lemma \ref{lem:rational-points} and in section \ref{sec:wf-lie-algebra}.
\end{remark}

\subsection{Wave front sets}\label{sec:wf-lie-algebra}
Transfer the map $\mathscr N$ to a map $\mathfrak g^*/G \to \mathcal N(\mathfrak g^*)/G$ using the identification between $\mathfrak g$ and $\mathfrak g^*$ induced by the bilinear form $B$.
%Call an orbit $\mathcal O^*$ good if the corresponding orbit of $\mathfrak g$ contains a good element.
\begin{theorem}
    \label{thm:ratwfungraded}
    Let $\mathcal O^*\in \mathfrak g^\Fr/G^\Fr$ be a co-adjoint orbit.
    Then 
    $$\overline{\WF}(\chi_{\mathcal O^*}) \subset \overline{\mathscr N(\mathcal O^*)}$$
    and 
    $$\overline{\WF}(\chi_{\mathcal O^*}) \cap \mathscr N(\mathcal O^*)^\Fr \ne\emptyset.$$
    %
    %there exists a $\mathcal O'^*\subset \mathscr N(\mathcal O^*)^\Fr$ such that 
    
\end{theorem}
\begin{proof}
    The containment follows from proposition \ref{prop:slice} and proposition \ref{prop:dilatedsuppungraded}.
    
    It remains to show that 
    $$\overline{\WF}(\chi_{\mathcal O^*}) \cap \mathscr N(\mathcal O^*)^\Fr \ne\emptyset.$$
    %
    
    %Let $\mathcal O'\in \mathcal N(\mathfrak g)/G$.
    %We have that 
    %\begin{equation}
        %\label{eq:rawf}
        %\langle f,\Gamma_{\mathcal O^*}\rangle_{\mathfrak g_i} = \langle \FT(f), \FT(\Gamma_{\mathcal O^*})\rangle_{\mathfrak g_i^*} = \langle1_{\mathcal O'^*},\FT(\Gamma_{\mathcal O^*})\rangle.
    %\end{equation}
    %By Equation \ref{eq:ftgggg} we have that Equation \ref{eq:rawf} is non-zero if and only if $\mathcal O'^* \cap \Sigma_{\mathcal O^*}\ne \emptyset$.
    %By Proposition \ref{prop:slice}, $\mathcal O'^* \cap \Sigma_{\mathcal O^*}\ne \emptyset$ implies that $\mathcal O^* \subseteq \overline{\widetilde{G}.\mathcal O'^*}$ as required.
    %
    Let $\alpha \in \mathcal O^*$ and $\phi\in \Theta(\eta_B(\alpha))^\Fr$.
    By definition of $\Theta(\eta_B(\alpha))^\Fr$, we have $\eta_B(\alpha) \in \mathfrak s^\Fr_\phi$ and $G. \phi(e) = \mathscr N(\eta_B(\alpha))$.
    Let $\mathcal O' = G^\Fr.\phi(e)$ and let $\mathcal O'^*$ denote the corresponding orbit of $\mathfrak g^*$.
    By Equation \ref{eq:ftgggg} it follows that $\langle \chi_{\mathcal O^*}, \Gamma_{\mathcal O'^*}\rangle \ne 0$ as required.
    
    %Now suppose $f$ is an arbitrary $G_0$-invariant character of $\mathfrak g_i$.
    %Then $f = \sum_{\mathcal O'^*\subseteq \check V/G}c_{\mathcal O'^*}(f)\chi_{\mathcal O'^*}$ for some $c_{\mathcal O'^*}(f)\in \mathbb N_0$.
    %Thus $\langle f,\Gamma_{\mathcal O^*}\rangle_{\mathfrak g_i}\ne 0$ if and only if $\langle \FT(1_{\mathcal O'^*}),\Gamma_{\mathcal O^*}\rangle_{\mathfrak g_i}\ne 0$ for some $\mathcal O'^*$ with $c_{\mathcal O'^*}(f)\ne 0$.
    %Since 
    %%
    %$$\WF^{git}(f) = \bigcup_{\mathcal O'^*:c_{\mathcal O'^*}(f)\ne 0} \WF^{git}(\chi_{\mathcal O'^*})$$
    %%
    %the result follows from the first case considered.
\end{proof}

%\begin{corollary}
    %Let $f:\mathfrak g\to \mathbb C$ be a $G$-invariant character of $\mathfrak g$. Then
    %%
    %$$\WF(f) = \bigcup_{\mathcal O^*\in \WF^{rat}(f)}\overline{\mathbb G_0.\mathcal O^*}.$$
    %%
%\end{corollary}
%\begin{proof}
    %We have that $f = \sum_{\mathcal O'^*\subseteq \check V/G}c_{\mathcal O'^*}(f)\chi_{\mathcal O'^*}$ for some $c_{\mathcal O'^*}(f)\in \mathbb N_0$.
    %The result then follows from Lemma \ref{lem:asywfadditive} and Lemma \ref{lem:ratwfadditive}.
%\end{proof}
%
\section{The wave front set of positive depth representations}
\label{sec:wf}
\subsection{Preliminaries}
Let $\mathbb G$ be a reductive group defined over $F$, $\textfrak g$ denote the Lie algebra of $\mathbb G$ and let $G=\mathbb G(F), \mathfrak g = \textfrak g(F)$ and $G^{un} = \mathbb G(\Fun), \mathfrak g^{un} = \textfrak g(\Fun)$.
%Let $L/F^{un}$ be the minimal Galois extension of $K$ over which $\mathbb G$ splits, and normalise the valuation on $L$ so that $\nu(K^\times) = \mathbb Z$.
Let $\Fr\in \Gal(\Fun/F)$ be a lift of Frobenius.
Let $\mu$ (resp. $\mu^{\mathfrak g}$) denote a fixed Haar measure on $G$ (resp. $\mathfrak g$).
For the main results of this section we will need the residue characteristic $p$ of $F$ to be larger than some constant depending on the absolute root datum of $\mathbb G$ and the ramification index of $F/\mathbb Q_p$. 
We refer to hypotheses \ref{hyp:exp}, \ref{hyp:p}, and \ref{hyp:bilinear-form} for the precise requirements.
%We do not attempt to find an optimal set of conditions on $p$. 
%Instead we indicate precisely when an assumption on the residue characteristic is made and provide a reference that contains a disucssion on the precise bounds needed on $p$.

Let $\mathscr B(\mathbb G,F)$ (resp. $\mathscr B(\mathbb G,F^{un})$) denote the (enlarged) Bruhat-Tits building for $G$ (resp. $G^{un}$).
For $x\in \mathscr B(\mathbb G,F)$ (resp. $x\in \mathscr B(\mathbb G,\Fun)$) and $r\ge 0$ let $G_{x,r}$ (resp. $G^{un}_{x,r}$) denote the Moy-Prasad filtration subgroups for $G$ (resp. $G^{un}$) associated to the point $x$.
Similarly for $r\in \mathbb R$, let $\mathfrak g_{x,r}, \mathfrak g^*_{x,r}, \mathfrak g^{un}_{x,r},\mathfrak g^{un,*}_{x,r}$ denote the Moy-Prasad filtration subgroups of the Lie algebra and the linear dual of the Lie algebra.
Let $G_{x,r^+}, G^{un}_{x,r^+},\mathfrak g_r,\mathfrak g_{r^+}$ etc. denote the usual objects.
Recall that $\mathfrak g^{*}_{x,-r}$ can be characterised as 
$$\mathfrak g^*_{x,-r} = \{\alpha\in \mathfrak g^*: \alpha(X) \in \mathfrak p,\ \forall X\in \mathfrak g_{x,r^+}\}.$$

\begin{hypothesis}\label{hyp:exp}
    Either
    \begin{enumerate}
        \item the exponential function $\exp$ converges on $\mathfrak g_{0^+}$, or 
        \item there is a suitable mock-exponential function defined on $\mathfrak g_{0^+}\to G_{0^+}$.
    \end{enumerate}
    We refer to \cite[Lemma 3.2]{barmoy} for a sufficient condition to guarantee the convergence of the exponential map. 
    We refer to \cite[Remark 3.2.2]{debacker2} for a discussion on mock exponential maps.
    In either case we write
    $$\exp:\mathfrak g_{0^+}\to G_{0^+}$$
    for the chosen function.
    
    %Then the exponential map converges on $\mathfrak g_{0^+}$.
    %We write $\exp:\mathfrak g_{0^+}\to G_{0^+}$ for the exponential function.
    %\cite[Hypothesis 3.2.1]{debacker2} holds.
\end{hypothesis}

\begin{definition}
    Let $x\in \mathscr B(\mathbb G,F)$,
    \begin{enumerate}
        \item Define
        $$\underline{\mathfrak g}^{un}_{x,r} := \mathfrak g^{un}_{x,r}/\mathfrak g^{un}_{x,r^+}, \quad \underline{\mathfrak g}^{un,*}_{x,r} := \mathfrak g^{un,*}_{x,r}/\mathfrak g^{un,*}_{x,r^+}.$$
        Since $x\in \mathscr B(\mathbb G,F)$ it is $\Gal(\Fun/F)$-fixed and
        $$\underline{\mathfrak g}_{x,r} = \mathfrak g_{x,r}/\mathfrak g_{x,r^+}, \quad \underline{\mathfrak g}^{*}_{x,r} = \mathfrak g^{*}_{x,-r}/\mathfrak g^{*}_{x,r^+}.$$
        For an element $X\in \mathfrak g_{x,r}$ we write $\underline X$ for its image in $\underline {\mathfrak g}_{x,r}$.
        Note that
        $$\underline{\mathfrak g}_{x,r} = \underline{\mathfrak g}_{x,r}^{un,\Fr}, \quad \underline{\mathfrak g}^*_{x,r} = \underline{\mathfrak g}_{x,r}^{un,*,\Fr}.$$
        \item The natural pairing $\langle-,-\rangle:\mathfrak g^{*}\times \mathfrak g\to F$ descends to a perfect pairing
        $$\langle-,-\rangle:\underline{\mathfrak g}^{*}_{x,-r}\times \underline{\mathfrak g}_{x,r}\to \mathbb F_q.$$
        Thus we obtain a natural identification between $(\underline{\mathfrak g}_{x,r})^*$ and $\underline{\mathfrak g}^{*}_{x,-r}$.
        \item The character $\chi$ induces an isomorphism 
        $$\underline{\mathfrak g}^{*}_{x,-r} \xrightarrow{\sim} (\underline{\mathfrak g}_{x,r})^\wedge, \quad \alpha \mapsto \chi(\langle \alpha, - \rangle)$$
        between $\underline{\mathfrak g}^{*}_{x,-r}$ and the Pontryagin dual of $\underline{\mathfrak g}_{x,r}$.
        \item Recalling the identification in $(2)$, let $\FT_{x,r}$ denote the Fourier transform from $C(\underline{\mathfrak g}_{x,r})$ to $C(\underline{\mathfrak g}^{*}_{x,-r})$ and $\FT^*_{x,-r}$ denote the Fourier transform from $C(\underline{\mathfrak g}^{*}_{x,-r})$ to $C(\underline{\mathfrak g}_{x,r})$.
        \item For a compactly supported smooth function $f:\mathfrak g\to \mathbb C$ define
        $$\mathrm{FT}(f):\mathfrak g^* \to \mathbb C, \quad \alpha \mapsto \int_{\mathfrak g}\chi(\alpha(x))f(x)dx.$$
        Define the inner product $\langle-,-\rangle$ on $C_c^{\infty}(\mathfrak g)$ to be 
        $$\langle f, g \rangle =  \int_{\mathfrak g}\overline{f(x)}g(x)dx, \quad f,g\in C_c^\infty(\mathfrak g).$$
        Similarly, define the Fourier transform and inner product for compactly supported smooth functions on $\mathfrak g^*$.
        \item Let $\underline G^{un}_{x}:=G^{un}_{x,0}/G^{un}_{x,0^+}$.
        This is the $\overline{\mathbb F}_q$-points of a connected reductive group defined over $\mathbb F_q$ and $\underline G_x^{un,\Fr} = G_{x,0}/G_{x,0^+}$.
        \item For $r>0$, define $\underline G^{un}_{x,r}:=G^{un}_{x,r}/G^{un}_{x,r^+}$ and $\underline G_{x,r}:=G_{x,r}/G_{x,r^+}$.
        The exponential map $\exp$ descends to an isomorphism of abelian groups
        $$\exp_{x,r}:\underline{\mathfrak g}^{un}_{x,r} \xrightarrow{\sim} \underline G^{un}_{x,r}$$
        that restricts to an isomorphism $\underline{\mathfrak g}_{x,r}\to \underline G_{x,r}$.
        \item For a function $f:\underline{\mathfrak g}_{x,r}\to \mathbb C$ let $L^{\mathfrak g}_{x,r}(f):\mathfrak g\to \mathbb C$ denote the function on $\mathfrak g$ obtained by lifting $f$ along $\mathfrak g_{x,r}\to \underline{\mathfrak g}_{x,r}$ and extending by zero outside of $\mathfrak g_{x,r}$.
        Similarly define $L^{\mathfrak g^*}_{x,r}$ for functions on $\underline{\mathfrak g}^{*}_{x,r}$ and $L^{G}_{x,r}$ for functions on $\underline G_{x,r}$.
        %\item For $r\in \mathbb R$, the adjoint action of $G^{un}$ on $\mathfrak g^{un}$ descends to an action of $\underline G_{x}$ on $\underline{\mathfrak g}_{x,r}$, and similarly we obtain an action of $\underline G_x$ on $\underline{\mathfrak g}^*_{x,r}$ from the coadjoint action.
        %The analogous statement for $\underline G_x$, $\underline{\mathfrak g}_{x,r}^\Fr$ and $\underline{\mathfrak g}^{*,\Fr}_{x,r}$ is also true.
        \item Define the Lie algebra
        $$\overline{\mathfrak g}^{un}_x := \bigoplus_{r\in \mathbb R/\mathbb Z}\underline{\mathfrak g}^{un}_{x,r}$$
        with Lie bracket 
        $$[-,-]:\underline{\mathfrak g}^{un}_{x,r}\times \underline{\mathfrak g}^{un}_{x,s}\to \underline{\mathfrak g}^{un}_{x,r+s}$$
        defined as in \cite[\S4.2]{debacker}.
        The group $\underline G^{un}_x$ acts on $\underline{\mathfrak g}^{un}_{x,r}$ for all $r$ and has Lie algebra $\underline{\mathfrak g}^{un}_{x,0}$.
        Moreover the actions are compatible with $\Fr$.
    \end{enumerate}
\end{definition}

%\begin{lemma} 
    %Let $x\in \mathscr B(\mathbb G, F)$ be a rational point of order $m$ (c.f. \cite[\S3.2]{fintzen}) with $(q,m) = 1$.
%
    %Let $H$ be a connected reductive group defined over $\overline{\mathbb F}_q$ with Lie algebra $\mathfrak h$, and with the same absolute root datum as $\mathbb G$.
    %
    %Then there is an order $m$ automorphism $\vartheta$ and a Frobenius endomorphism $\Fr$ of $H$ such that
    %\begin{enumerate}
        %\item $\underline G_x \simeq H^\vartheta$,
        %\item the $\mathbb Z/m$-grading 
        %%
        %$$\bigoplus_{i\in \mathbb Z/m} \mathfrak h_i$$
        %%
        %induced by $\vartheta$ is isomorphic as a graded Lie algebra to $\overline{\mathfrak g}_x$,
        %\item $\Fr\circ \vartheta = \vartheta^q\circ \Fr$, and the action of $\Fr$ coincides with the action of Frobenius on $\underline{\mathfrak g}_{x,r}$ and $\underline{G}_x$.
    %\end{enumerate}
%\end{lemma}
%\begin{proof}
    %This was proved for tamely ramified semisimple groups in \cite[Theorem 4.1]{reeder-yu} and in \cite{fintzen} for tamely ramified reductive groups.
%\end{proof}
%Note that thanks to the work of Debacker in \cite{debacker}, the construction of the 4G representations still makes sense in the full generality of $(\underline G_x,\overline {\mathfrak g}_x)$ so we won't impose the restriction that $L$ is tamely ramified over $F^{un}$ yet.

%We normalise the Haar measure so that the measure of $\mathfrak g_{x,0}$ for $x$ a special point is $1$.
\begin{hypothesis}\label{hyp:p}
    The hypotheses in \cite[\S4]{debacker} hold for $\mathbb G$.
    In particular the data $\underline G^{un}_x, \overline{\mathfrak g}^{un}_x,\sigma$ satisfy the hypotheses in hypothesis \ref{hyp:graded}.
\end{hypothesis}

\begin{lemma}\cite[\S3.1]{debacker2}
    Let $r\le s$.
    Then $f\in C(\mathfrak g_{x,r}/\mathfrak g_{x,s})$ if and only if $\FT(f)\in C(\mathfrak g^*_{x,(-s)^+}/\mathfrak g^*_{x,(-r)^+})$.
\end{lemma}

The following lemma is an easy calculation.
\begin{lemma}\label{lem:ft}
    For functions $f:\underline{\mathfrak g}_{x,r}\to \mathbb C$
    $$\FT(L^{\mathfrak g}_{x,r}(f)) = C_{x,r}L^{\mathfrak g^*}_{x,-r}(\FT_{x,r}(f)),$$
    where $C_{x,r}$ is the positive constant $\mu^{\mathfrak g}(\mathfrak g_{x,r^+})q^{N_{x,r}/2}$ and $N_{x,r} = \dim_{\mathbb F_q}(\underline{\mathfrak g}_{x,r})$.
\end{lemma}

\begin{definition}
    For $r\in \mathbb R$, let 
    $$\mathscr L_{x,r}:\mathcal N(\underline{\mathfrak g}_{x,r})\to \mathcal N(\mathfrak g)/G$$
    denote the lifting map of nilpotent orbits introduced by Debacker in \cite{debacker}.
    
    This map is constant on $\underline G_x$-orbits and for every orbit $\mathcal O\subseteq \mathcal N(\mathfrak g)/G$ there is an $x\in \mathscr B(\mathbb G,F)$ and an orbit $\underline {\mathcal O}$ of $\mathcal N(\underline{\mathfrak g}_{x,r})/\underline G_x$ that lifts to $\mathcal O$.

    We remark that the map can be extended to $\mathcal N(\underline{\mathfrak g}^{un}_{x,r})\to \mathcal N(\mathfrak g^{un})/G^{un}$ since $\Fun$ is the union of all finite unramified extensions of $F$.
\end{definition}

%\begin{definition}
    %For $0 \le s\le r$ say $s \le^+ l$ if $s\le r$ when $s>0$ and $r>0$ when $s = 0$.
%\end{definition}
\subsection{Main theorem}
%Recall that $\rho(\pi)$ is always rational.

%Let
%%
%$$\Theta_\pi(\exp(X)) = \sum_{\mathcal O\subseteq \mathcal N(\mathfrak g)/G}c_{\mathcal O}(\pi)\FT(\mu_{\mathcal O})(X)$$
%%
%be the local character expansion for $\pi$.
%Since $\pi$ has depth $\le r$, the local character expansion is valid on $\mathfrak g_{r^+}$.
\begin{lemma}
    Let $(\pi,V)$ be a smooth irreducible representation of depth $\rho(\pi)$. 
    Let $r\ge \rho(\pi)$ and suppose $r>0$.
    
    For $x\in \mathscr B(\mathbb G,F)$ let $\underline \pi_{x,r}$ denote the finite dimensional $\underline G_{x,r}$-representation $V^{G_{x,r^+}}$.
    Let 
    $$\underline \theta_{\pi,x,r}:\underline G_{x,r}\to \mathbb C$$
    denote the character of this representation and let $\theta_{\pi,x,r} := \underline \theta_{\pi,x,r}\circ \exp_{x,r}$ denote the corresponding character of $\underline{\mathfrak g}_{x,r}$.
    Then $\theta_{\pi,x,r}$ admits a decomposition
    $$\theta_{\pi,x,r} = \sum_{\underline{\mathcal O^*}\in \underline{\mathfrak g}^*_{x,-r}}m_{x,r,\underline{\mathcal O}^*}(\pi)\chi_{\underline{\mathcal O^*}}$$
    into irreducible $\underline G_x$-invariant characters.
\end{lemma}
\begin{proof}
    Since $\underline G_{x,r}$ is normal in $G_{x,0}/G_{x,r^+}$ and $G_{x,0^+}/G_{x,r^+}$ acts trivially on $\underline G_{x,r}$, we get that $\theta_{\pi,x,r}$ is a $\underline G_x$-invariant character of $\underline{\mathfrak g}_{x,r}$.
    The result then follows by lemma \ref{lem:coadjcharexp} and lemma \ref{lem:irrchar}.
    
\end{proof}
%For each $\alpha \in \mathcal X_{x,r}$ let $m_{\alpha}(\underline\pi_{x,r})$ denote the multiplicity of the corresponding character in Equation \ref{eq:res}.
%After identifying the Pontryagin dual of $\underline H_{x,r}$ with $\underline{\mathfrak g}^*_{x,-r}$ via the character $\chi$ we obtain a subset $\mathcal X_{x,r}\subseteq \underline{\mathfrak g}^*_{x,-r}$ corresponding to the characters arising in the restriction in Equation \ref{eq:res}.

\begin{hypothesis}\label{hyp:bilinear-form}
    The Lie algebra $\mathfrak g$ admits an $F$-valued, non-degenerate, $G$-equivariant, symmetric, bilinear form 
    $$B(-,-):\mathfrak g\times \mathfrak g\to F.$$
    We refer to \cite[Proposition 4.1]{adler-roche} for a set of sufficient conditions for this to hold.
    
    Such a form $B$, induces an isomorphism
    $$\eta_B:\mathfrak g^*\to \mathfrak g$$
    that restricts to an isomorphism
    $$\eta_B:\mathfrak g^*_{x,r} \to \mathfrak g_{x,r}$$
    and descends to an isomorphism
    $$\eta_{B,x,r}:\underline{\mathfrak g}^*_{x,r}\to \underline{\mathfrak g}_{x,r}.$$
    Under this identification $\FT^*_{x,-r} = \FT_{x,-r}$.
    We also use $\eta_B$ to yield lifting maps
    \begin{align*}
        &\mathscr L_{x,r}:\mathcal N(\underline{\mathfrak g}^{*}_{x,r})/\underline G_x\to \mathcal N(\mathfrak g^*)/G\\
        &\mathscr L^{un}_{x,r}:\mathcal N(\underline{\mathfrak g}^{un,*}_{x,r})/\underline G^{un}_x\to \mathcal N(\mathfrak g^{un,*})/G^{un}.
    \end{align*}
\end{hypothesis}

\begin{definition}\label{d:test-function}
    Let $r>0$.
    For $\underline{\mathcal O}^*\in \mathcal N(\underline{\mathfrak g}^{*}_{x,-r})$, define
    \begin{align}
        f^{\mathfrak g}_{x,r,\underline{\mathcal O}^*} &:= L^{\mathfrak g}_{x,r}(\Gamma_{\underline {\mathcal O}^*}) \in C_c^\infty(\mathfrak g), \\
        f_{x,r,\underline{\mathcal O}^*} &:= L^{G}_{x,r}(\Gamma_{\underline {\mathcal O}^*}\circ \exp_{x,r}^{-1}) \in C_c^\infty(G).
    \end{align}
    Note that 
    $$f_{x,r,\underline{\mathcal O}^*}\circ \exp = f^{\mathfrak g}_{x,r,\underline{\mathcal O}^*}$$
    on $\mathfrak g_{r}$.
\end{definition}

\begin{lemma}
    \label{lem:eval}
    Let $r>0$ and $\underline{\mathcal O}^*\in \mathcal N(\underline{\mathfrak g}^{*}_{x,-r})$.
    Then
    \begin{enumerate}
        \item $\supp(f_{x,r,\underline{\mathcal O}^*}) \subset G_{r^+}$,
        \item for any smooth admissible representation $(\pi,V)$ of $G$,
        \begin{align*}
            %\label{eq:evalgggg}
            \Theta_\pi(f_{x,r,\underline{\mathcal O}^*}) &= \mu(G_{x,r})\langle \theta^*_{\pi,x,r},\Gamma_{\underline{\mathcal O}^*}\rangle_{\underline{\mathfrak g}_{x,r}}
        \end{align*}
        where $\theta_{\pi,x,r}^*$ denotes the complex conjugate of $\theta_{\pi,x,r}$.
    \end{enumerate}
\end{lemma}
\begin{proof}
    Since $\Gamma_{\underline{\mathcal O}^*}$ has support contained in $\mathcal N(\underline{\mathfrak g}_{x,r})$, by \cite[Proposition 6.3]{moyprasad}, the function $f_{x,r,\underline{\mathcal O}^*}$ is supported in $G_{r+}$.

    For $(2)$ note that
    \begin{align*}
        \Theta_\pi(f_{x,r,\underline{\mathcal O}^*}) &= \mu(G_{x,r^+})\sum_{\underline g\in \underline{G}_{x,r}}\underline {\theta}_{\pi,x,r}(\underline g)\Gamma_{\underline{\mathcal O}^*}(\exp_{x,r}^{-1}(\underline g)) \\
        &= \mu(G_{x,r})\langle\theta^*_{\pi,x,r},\Gamma_{\underline{\mathcal O}^*}\rangle_{\underline{\mathfrak g}_{x,r}}.
    \end{align*}
\end{proof}

\begin{lemma}\label{lem:conj}
    Let $x\in \mathscr B(\mathbb G,F)$ and
    $$E\in \mathfrak g_{x,r}, \quad H\in \mathfrak g_{x,0}, \quad F\in \mathfrak g_{x,-r}$$
    be an $\mathfrak {sl}_2$-triple.
    Let $r_0 = r < r_1 < r_2 < \dots$ denote the jumps in the Moy-Prasad filtrations $\mathfrak g_{x,s}$ for $s\ge r$.
    If 
    $$X \in E + \mathfrak g_{x,r_0}(\le0) + \mathfrak g_{x,r_i}$$
    where $i>0$, then there exists $g\in G_{x,r_i-r_0}$ such that
    $$\Ad(g)(X) = E + \mathfrak g_{x,r_0}(\le0) + \mathfrak g_{x,r_{i+1}}.$$
\end{lemma}
\begin{proof}
    Let $\lambda$ be an adapted cocharacter for the $\mathfrak {sl}_2$-triple $\underline E,\underline H,\underline F$. 
    It induces a bi-grading 
    $$\bigoplus_{s\in \mathbb R/\mathbb Z,k\in \mathbb Z} \underline{\mathfrak g}_{x,s}(k)$$
    coming from the weights of $\Ad(\lambda)$.
    The action of $\ad(H)$ also induces a $\mathbb Z$-grading on $\mathfrak g_{x,s}$ for all $s$ and the $k$-weight space surjects onto $\underline{\mathfrak g}_{x,s}(k)$.
    
    Let $X = E+Z+W$ where $Z\in \mathfrak g_{x,r_0}(\le0)$ and $W\in \mathfrak g_{x,r_i}$.
    
    Write $W = W(\le 0) + W(>0)$ where $W(\le 0)\in \mathfrak g_{x,r_i}(\le 0)$ and $W(>0)\in \mathfrak g_{x,r_i}(>0)$.

    By a straightforward inductive argument on the highest non-zero weight of $\underline{W(>0)}$ there exists a $u\in \underline{\mathfrak g}_{x,r_i-r_0}$ such that
    $$[\underline X+ \underline Z,u] = \underline{W(>0)} \pmod{\underline{\mathfrak g}_{x,r_i}(\le0)}.$$
    Let $U\in \mathfrak g_{x,r_i-r_0}$ be a lift of $u$.
    Then $[U,W] \in \mathfrak g_{x,r_{i+1}}$ and so
    \begin{align*}
        \Ad(\exp(U))(E+Z+W) + \mathfrak g_{x,r_{i+1}}&= E+Z+W + [U,X+Z+W] + \mathfrak g_{x,r_{i+1}} \\
        &= \underline{E} + \underline Z + \underline{W(\le 0)} + \mathfrak g_{x,r_{i+1}} \\
        &\in E+ \mathfrak g_{x,r_0}(\le 0) + \mathfrak g_{x,r_{i+1}}.
    \end{align*}
    Since $\exp(U)\in G_{x,r_i-r_0}$ this provides the sought after element $g\in G_{x,r_i-r_0}$.
\end{proof}

\begin{lemma}
    \label{lem:liftingorder}
    Let $x\in \mathscr B(G,F)$ and
    $$E\in \mathfrak g_{x,r}, \quad H\in \mathfrak g_{x,0}, \quad F\in \mathfrak g_{x,-r}$$
    be an $\mathfrak {sl}_2$-triple.
    If $X\in \mathfrak g_{x,r}$ and its image $\underline X$ in $\underline{\mathfrak g}_{x,r}$ lies in 
    $$\underline {E} + \underline{\mathfrak g}_{x,r}(\le 0)$$
    then $X$ is conjugate to an element of $E+ \mathfrak g_{x,r}(\le 0)$.
\end{lemma}
\begin{proof}
    Repeatedly applying lemma \ref{lem:conj} yields sequences $(g_i\in G_{r_i-r_0})_{i=0}^\infty$ and $(Z_i\in \mathfrak g_{x,r_0}(\le0))_{i=0}^\infty$ such that
    $$\Ad(g_0\cdots g_i)X - Z_i \in \mathfrak g_{x,r_{i+1}}$$
    Since $(G_{x,r_i-r_0})_{i=0}^\infty$ form a neighbourhood system around the identity, $g_0\cdots g_i$ converges to some $\tilde g$.
    Since $\mathfrak g_{x,r_0}(\le0)$ is a closed subset of $\mathfrak g_{x,r_0}$ we moreover have that $\Ad(\tilde g)(X) = \lim_{i\to \infty}Z_i\in \mathfrak g_{x,r_0}(\le 0)$ as required.
\end{proof}

\begin{corollary}\label{cor:order-lift}
    Let $l\in \mathbb R$.
    %$\underline {\mathcal O}_1,\underline{\mathcal O}_2\in \mathcal N(V^\Fr_{x,l})/\underline G^\Fr_x$.
    \begin{enumerate}
        \item Let $\underline {\mathcal O}_1,\underline{\mathcal O}_2\in \mathcal N(\underline{\mathfrak g}_{x,l})/\underline G_x$. If $\underline{\mathcal O}_2\cap \Sigma_{\underline {\mathcal O}_1}\ne \emptyset$, then $\mathscr L_{x,l}(\underline {\mathcal O}_1)\le \mathscr L_{x,l}(\underline{\mathcal O}_2)$.
        \item Let $\underline {\mathcal O}_1,\underline{\mathcal O}_2\in \mathcal N(\underline{\mathfrak g}^{un}_{x,l})/\underline G^{un}_x$. If $\underline {\mathcal O}_1 \subset \overline{\underline{\mathcal O}_2}$, then $\mathscr L^{un}_{x,l}(\underline {\mathcal O}_1)\le \mathscr L^{un}_{x,l}(\underline{\mathcal O}_2)$.
    \end{enumerate}
\end{corollary}
\begin{proof}
    $(1)$ follows from \cite[Corollary 4.3.2]{debacker}, Lemma \ref{lem:liftingorder}, and \cite[Proposition 3.12]{barmoy}.

    $(2)$ follows from proposition \ref{prop:slice}.(1) and part $(1)$.
\end{proof}

\begin{proposition}
    \label{prop:orbitalint}
    Let $x\in \mathscr B(\mathbb G,F), r>0,\underline{\mathcal O}^*\in \mathcal N(\underline{\mathfrak g}^{*}_{x,-r})/\underline G_x$ and $\mathcal O'^*\in \mathcal N(\mathfrak g)/G$.
    We have
    \begin{enumerate}
        \item $\FT(\mu_{\mathcal O'^*})(f^{\mathfrak g}_{x,r,\underline{\mathcal O}^*}) \ne 0$ only if $\mathscr L_{x,-r}(\underline{\mathcal O}^*) \le \mathcal O'^*$;
        \item $\FT(\mu_{\mathcal O'^*})(f^{\mathfrak g}_{x,r,\underline{\mathcal O}^*}) \ne 0$ if $\mathscr L_{x,-r}(\underline{\mathcal O}^*) = \mathcal O'^*$.
    \end{enumerate}
\end{proposition}
\begin{proof}
    $(1)$ Suppose $\FT(\mu_{\mathcal O'^*})(f^{\mathfrak g}_{x,r,\underline{\mathcal O}^*}) \ne 0$.
    Then $\mu_{\mathcal O'^*}(\FT(f^{\mathfrak g}_{x,r,\underline{\mathcal O}^*}))\ne 0$ and so the support of $\FT(f^{\mathfrak g}_{x,r,\underline{\mathcal O}^*})$ must intersect $\mathcal O'^*$.
    But by lemma \ref{lem:ft}, $\FT(f^{\mathfrak g}_{x,r,\underline{\mathcal O}^*}) = C_{x,r}L_{x,-r}\FT_{x,r}(\Gamma_{\underline{\mathcal O}^*})$.
    Thus there is some element $\beta\in \supp(\FT_{x,r}(\Gamma_{\underline{\mathcal O}^*}))\subseteq \underline{\mathfrak g}^{*}_{x,-r} = \mathfrak g^*_{x,-r}/\mathfrak g^*_{x,(-r)^+}$ viewed as a coset in $\mathfrak g^*_{x,-r}$ that intersects $\mathcal O'^*$.
    %Since $\mathcal O'$ is nilpotent, $v$ is nilpotent and $\mathscr L_{x,-r}(v) \le \mathcal O'$.
    By equation \eqref{eq:ftgggg}, $\underline G_x.\beta \cap \Sigma_{\underline{\mathcal O}^*}\ne \emptyset$.
    By lemma \ref{lem:liftingorder} and \cite[Proposition 3.1.2]{moyprasad}, $\mathscr L_{x,-r}(\underline{\mathcal O}^*) \le \mathcal O'^*$ as required.
    $(2)$ Follows since $\underline {\mathcal O}^*\subseteq \supp(\FT_{x,r}(\Gamma_{\underline{\mathcal O}^*}))$.
\end{proof}

%\begin{definition}
    %\label{def:shallow}
    %Let $(\pi,V)$ be a smooth irreducible representation, $x\in \mathscr B(\mathbb G,F)$ and $r>0$. 
    %Define
    %%
    %$$\WF_{x,r}(\pi) = \mathscr L_{x,-r}(\overline{\WF}(\theta_{\pi,x,r}))\subseteq \mathcal N(\mathfrak g)/G.$$
    %%
%\end{definition}

Recall that every irreducible representation $(\pi,V)$ has rational depth.
\begin{theorem}
    \label{thm:shallow}
    Let $(\pi,V)$ be a smooth irreducible representation of $G$ with depth $\rho(\pi)$ and let $r\in \mathbb R_{>0}$, $r\ge \rho(\pi)$.

    Let $\mathscr C\subset \mathscr B(\mathbb G,F)$ be a subset such that every nilpotent orbit is in the image of $\mathscr L_{x,-r}$ for some $x\in \mathscr C$.
    
    Then 
    $$\WF(\pi) = \max\bigcup_{x\in \mathscr C}\mathscr L_{x,-r}(\overline{\WF}(\theta^*_{\pi,x,r})) = \max\bigcup_{x\in \mathscr C}-\mathscr L_{x,-r}(\cone(\supp(\FT_{x,r}(\theta_{\pi,x,r})))).$$
\end{theorem}
%\begin{remark}
    %The main difference with Theorem \ref{thm:deep} is that $r=\rho(\pi)$ is allowed.
%\end{remark}
\begin{proof}
    %Other than this the proof is essentially the same as the proof of Theorem \ref{thm:deep}, but we give it for completeness.
    Let
    $$\Theta_\pi = \sum_{\mathcal O^*\in \mathcal N(\mathfrak g^*)/G}c_{\mathcal O^*}(\pi)\FT(\mu_{\mathcal O^*})\circ \exp^*$$
    where $c_{\mathcal O^*}(\pi)\in \mathbb C$, denote the local character expansion for $\Theta_\pi$. 
    Since the functions $f_{x,r,\underline{\mathcal O}}$ are suppored on $G_{r+}$, they are defined on the domain of validity for the local character expansion and so
    \begin{align*}
       \Theta_\pi(f_{x,r,\underline{\mathcal O}^*}) &= \sum_{\mathcal O'^*\in \mathcal N(\mathfrak g^*)/G}c_{\mathcal O'^*}(\pi)\FT(\mu_{\mathcal O'^*})(f^{\mathfrak g}_{x,r,\underline{\mathcal O}^*}) \\
       &= \sum_{\substack{\mathcal O'^*\in \mathcal N(\mathfrak g^*)/G:\\ \mathscr L_{x,-r}(\underline{\mathcal O}^*)\le \mathcal O'^*}}c_{\mathcal O^*}(\pi)\FT(\mu_{\mathcal O^*})(f^{\mathfrak g}_{x,r,\underline{\mathcal O}^*})
    \end{align*}
    where the last equality follows from proposition \ref{prop:orbitalint}.(1).
    
    We will first show that if $x\in \mathscr C$ and $\underline{\mathcal O}^*\in \mathcal N(\underline{\mathfrak g}_{x,-r})/\underline G_x$ is such that $\mathscr L_{x,-r}(\underline{\mathcal O}^*) \in \WF(\pi)$, then $\underline{\mathcal O}^*\in \overline{\WF}(\theta^*_{\pi,x,r})$.

    Since $\mathcal O^*:=\mathscr L_{x,-r}(\underline {\mathcal O}^*)\in \WF(\pi)$, we have that $c_{\mathcal O'^*}(\pi) = 0$ for all orbits larger than $\mathcal O^*$ and so
    $$\Theta_\pi(f_{x,r,\underline{\mathcal O}^*}) = c_{\mathcal O^*}(\pi)\FT(\mu_{\mathcal O^*})(f^{\mathfrak g}_{x,r,\underline{\mathcal O}^*}).$$
    By lemma \ref{prop:orbitalint}.(2)
    $$\FT(\mu_{\mathcal O^*})(f^{\mathfrak g}_{x,r,\underline{\mathcal O}^*}) \ne 0.$$
    Therefore by lemma \ref{lem:eval}.(2)
    $$0 \ne c_{\mathcal O^*}(\pi)\FT(\mu_{\mathcal O^*})(f_{x,r,\underline{\mathcal O}^*}) = \Theta_\pi(f_{x,r,\underline{\mathcal O}^*}) = \mu(G_{x,r})\langle \overline{\theta}_{\pi,x,r},\Gamma_{\underline{\mathcal O}}\rangle_{\underline{\mathfrak g}_{x,r}}.$$
    Thus $\underline{\mathcal O}\in \overline{\WF}(\theta^*_{\pi,x,r})$.

    Finally, if $\underline{\mathcal O}^*\in \overline\WF(\theta^*_{\pi,x,r})$ then
    $$0 \ne \mu(G_{x,r})\langle \theta^*_{\pi,x,r},\Gamma_{\underline{\mathcal O}^*}\rangle_{\underline{\mathfrak g}_{x,r}} = \Theta_\pi(f_{x,r,\underline{\mathcal O}^*}).$$
    Thus there must be some $\mathcal O'^*\in \WF(\pi)$ such that $\mathscr L_{x,-r}(\underline{\mathcal O}^*)\le \mathcal O'^*$. 
    This implies the first equality.

    For the second equality simply note that $\overline\chi_{\underline{\mathcal O}^*} = \chi_{-\underline{\mathcal O}^*}$ and then apply lemma \ref{lem:wf-cone-relation}.

    %Suppose $\mathcal O\in \max_{x\in \mathscr C}\WF^{\shallow}_{x,r}(\pi)$.
    %Then $\mathcal O = \mathscr L_{x,-r}(\underline{\mathcal O})$ for some $x\in \mathscr C$ and $\underline{\mathcal O}\in \mathcal N(V_{x,-r})/\underline G_x$ such that $\langle \theta_{\pi,x,r},\Gamma_{\underline{\mathcal O}}\rangle_{\underline{\mathfrak g}_{x,r}} \ne 0$.
    %Let $\mathscr G_{\ge \mathcal O}(\pi)$ be the set of orbits $\mathcal O'\ge \mathcal O$ with $c_{\mathcal O'}(\pi) \ne 0$.
    %The set is non-empty since otherwise 
    %%
    %$$0\ne \langle \underline{\theta}_{\pi,x,r},\Gamma_{\underline{\mathcal O}}\rangle_{\underline{\mathfrak g}_{x,r}} = \Theta_\pi(f_{x,r,\underline{\mathcal O}}\circ \exp^{-1}) = \sum_{\mathcal O'\ge \mathcal O}c_{\mathcal O'}(\pi)\FT(\mu_{\mathcal O'})(f_{x,r,\underline{\mathcal O}}) = 0$$
    %%
    %which is a contradiction.
    %Let $\mathcal O'$ be a maximal element of $\mathscr G_{\ge \mathcal O}$.
    %Then clearly $\mathcal O'\in \WF(\pi)$ and $\mathcal O \le \mathcal O'$.
    %It follows that
    %%
    %$$\WF(\pi) = \max_{x\in \mathscr C}\WF^{\shallow}_{x,r}(\pi).$$
    %%
\end{proof}

\begin{corollary}\label{cor:z-n-graded}
    Let $(\pi,V)$ be a smooth irreducible representation of $G$ with depth $r>0$.

    Then $r=m/n\in \mathbb Q$ is rational and for any $x\in \mathscr B(\mathbb G,F)$, 
    $$\overline{\mathfrak g}_x^{un,r} := \bigoplus_{i\in\mathbb Z/n} \underline{\mathfrak g}^{un}_{x,ir}$$
    is a $\mathbb Z/n$-graded Lie algebra.

    Thus the wave front set of $\pi$ can be computed in terms of asymptotic cones of $\mathbb Z/n$-graded Lie algebras.
\end{corollary}

\begin{corollary}\label{cor:bound}
    Let $(\pi,V)$ be a smooth irreducible representation of $G$ with depth $\rho(\pi)>0$ and let $r\ge \rho(\pi)$.
    
    Let $\mathscr C\subset \mathscr B(\mathbb G,F)$ be a subset such that every nilpotent orbit is in the image of $\mathscr L_{x,-r}$ for some $x\in \mathscr C$.
    
    Then the wave front set of $(\pi,V)$ is bounded by
    $$\max\bigcup_{x\in \mathscr C}\mathscr L^{un}_{x,-r}(\overline{\cone}(\supp(\FT_{x,r}(\theta_{\pi,x,r})))) \quad \subseteq \mathfrak g^{un,*}.$$
\end{corollary}
\begin{proof}
    By proposition \ref{prop:slice}, for any subset $S\subset \underline{\mathfrak g}_{x,r}$ we have $\cone(S) \subset \overline\cone(S)$.
    Since $-1$ is a square in $\overline{\mathbb F}_q$ in fact $-\cone(S)\subset\overline\cone(S)$.
    The result then follows from theorem \ref{thm:shallow} and corollary \ref{cor:order-lift}.
\end{proof}

\section{Example}
\label{sec:adler}
In this section we demonstrate the usage of theorem \ref{thm:shallow} by computing the wave front set of certain positive-depth toral supercuspidal representations and recovering DeBacker--Reeder's condition of genericity \cite{debacker_reeder_2010}.

\begin{theorem}
    %\cite[Section 2.6]{debacker_reeder_2010}.
    Let $\mathbb G$ be a semisimple unramified group and $X$ be a regular good element of $\mathfrak g$ of depth $-r$ whose centraliser is an unramified minisotropic torus $S$ (c.f. \cite[\S2.6]{debacker_reeder_2010}).

    Suppose hypotheses \ref{hyp:exp}, \ref{hyp:p}, and \ref{hyp:bilinear-form} are met.
    
    Let $\Pi_X$ denote the set of supercuspidal representations attached to $X$ via the construction outlined in \cite[\S2.6]{debacker_reeder_2010}.

    Let $x$ be the unique fixed point of $S$ in $\mathscr B(\mathbb G,F)$, and let $\underline X$ denote the image of $X$ in $\underline{\mathfrak g}_{x,-r}$.

    Then
    \begin{enumerate}
        \item $r\in \mathbb Z$ and so $\overline{\mathfrak g}_x^r = \underline{\mathfrak g}_{x,0}$ is a $\mathbb Z/1$-graded (i.e. ungraded) Lie algebra,\label{ex:1}
        \item let $x_0$ be a hyperspecial point of $\mathscr B(\mathbb G,\Fun)$. Then $\underline{\mathfrak g}_{x,0}$ can be identified with the fixed points of an inner automorphism (coming from a semisimple element of $\underline G_{x_0}$) of $\underline{\mathfrak g}_{x_0,0}$,\label{ex:2}
        \item the geometric wave front set of any $\pi\in \Pi_X$ is a singleton whose weighted Dynkin diagram coincides with that of the saturation of $\mathscr N(\underline X)$ in $\underline{\mathfrak g}_{x_0,0}$,\label{ex:3}
        \item $\pi\in \Pi_X$ is generic if and only if $x$ is hyperspecial.\label{ex:4}
    \end{enumerate}
\end{theorem}
\begin{proof}
    \eqref{ex:1} By \cite[\S2.6]{debacker_reeder_2010} the condition that $S$ is unramified implies that $r\in \mathbb Z$.

    \eqref{ex:2} The point $x$ is a generalised $r$-facet and so in particular is a rational point of order $1$.
    The result then follows from \cite[Theorem 4.1]{reeder-yu}.
    
    \eqref{ex:3} We wish to apply theorem \ref{thm:shallow} and its corollary \ref{cor:z-n-graded}.
    Let $\mathscr C\subset \mathscr B(\mathbb G,F)$ denote a set of points including the point $x$, such that every co-adjoint nilpotent orbit lies in the image of $\mathscr L_{y,-r}$ for some $y\in \mathscr C$.
    
    We can discard any point $y\in \mathscr C$ $r$-associate to $x$ since $\mathscr L_{y,-r}$ is constant on the $r$-association classes.

    Let $\pi\in\Pi_X$.
    By  \cite[Lemma 2.6, Corollary 2.5]{debacker_reeder_2010}, 
    \[\pi^{G_{y,r^+}} = 0, \quad\text{for all }y\in \mathscr C, \ y\ne x.\]
    Thus by Theorem \ref{thm:shallow} 
    \begin{equation}\WF(\pi) = \max-\mathscr L_{x,-r}(\cone(\supp(\FT(\theta_{\pi,x,r})))).\end{equation}
    By the argument in \cite[Lemma 2.6]{debacker_reeder_2010} we see that 
    $$\supp(\FT(\theta_{\pi,x,r})) = \underline G_x.\underline X.$$

    By corollary \ref{cor:bound} and the first part of theorem \ref{thm:ratwfungraded}, the geometric wave front set is contained in the closure of 
    $$\mathbb G(\bar F).\mathscr L_{x,-r}(\mathscr N(\underline X)).$$
    Theorem \ref{thm:shallow} and the second part of theorem \ref{thm:ratwfungraded} implies that the geometric wave front set is in fact equal to this bound.

    Since $\mathfrak g^{un}$ and $\underline{\mathfrak g}_{x_0,0}$ have the same root data, there is a matching of geometric orbits by comparing weighted Dynkin diagrams.

    If $E\in \mathfrak g_{x,r},H\in \mathfrak g_{x,0},F\in \mathfrak g_{x,-r}$ is an $\mathfrak {sl}_2$-triple such that $\underline E \in  \mathscr N(\underline X)^\Fr$ then $\underline E,\underline H,\underline F$ is an $\mathfrak {sl}_2$-triple for $\mathscr N(\underline X)$.

    The pairings of $H$ and $\underline H$ with the character lattice of a maximal $\Fun$-split torus coincide.
    It follows that the geometric orbit of $\mathbb G(\bar F).E = \mathbb G(\bar F).\mathscr L_{x,-r}(\mathscr N(\underline X))$ has the same weighted Dynkin diagram as the saturation of $\mathscr N(\underline X)$ in $\underline {\mathfrak g}_{x_0,0}$.
    
    \eqref{ex:4} By part \eqref{ex:3}, $\pi$ is generic if and only if the saturation of $\mathscr N(\underline X)$ in $\underline {\mathfrak g}_{x_0,0}$ is the regular orbit.
    Since $\mathfrak g_{x,0}$ is the fixed point of an inner automorphism coming from a semisimple element $s$, it contains a regular nilpotent element of $\underline{\mathfrak g}_{x_0,0}$ if and only if $s$ is central.
    Now $s$ is central if and only if $\underline{\mathfrak g}_{x,0}$ is equal to $\underline{\mathfrak g}_{x_0,0}$.
    This is the case if and only if $x$ is hyperspecial.
\end{proof}

\begin{sloppypar} 
\printbibliography[title={References}] 
\end{sloppypar}

\end{document}